\documentclass[11pt]{amsart}
\usepackage{float}
\usepackage[usenames]{color}
\usepackage{graphicx}
\usepackage{amssymb}
\usepackage{amscd}
\usepackage{makecell}
\usepackage{amsmath}
\usepackage{mathtools}
\usepackage{tikz}
\usetikzlibrary{graphs}
\usetikzlibrary{arrows,decorations.markings,plotmarks,decorations.markings}
\usetikzlibrary{positioning}	
\usetikzlibrary{snakes} 

\usepackage{enumitem}

\theoremstyle{plain}
\newtheorem{thm}{Theorem}[section]
\newtheorem{lemma}[thm]{Lemma}
\newtheorem{cor}[thm]{Corollary}

\newtheorem{prop}[thm]{Proposition}

\theoremstyle{definition}
\newtheorem{defi}[thm]{Definition}

\newtheorem{example}[thm]{Example}

\newtheorem{conjecture}{Conjecture}

\DeclareMathOperator{\ind}{ind}

\DeclareMathOperator{\Ann}{Ann}

\DeclareMathOperator{\spin}{spin}

\DeclareMathOperator{\overlap}{overlap}
\DeclareMathOperator{\sing}{singularity}
\DeclareMathOperator{\DI}{DI}
\begin{document}

\title{On the non-vanishing condition for $A_\mathfrak{q}(\lambda)$ of $U(p,q)$ in the mediocre range}

\author{Du Chengyu}
\address[Du]{School of Mathematical Sciences, Soochow University, Suzhou 215006,
	P.~R.~China}
\email{cydu0973@suda.edu.cn}

\abstract{The modules $A_\mathfrak{q}(\lambda)$ of $U(p,q)$ can be parameterized by their annihilators and asymptotic supports, both of which can be identified using Young tableaux. Trapa developed an algorithm for determining the tableaux of the modules $A_\mathfrak{q}(\lambda)$ in the mediocre range, along with an equivalent condition to determine non-vanishing. The condition involves a combinatorial concept called the overlap, which is not straightforward to compute. In this paper, we establish a formula for the overlap and simplify the condition for ease of use. We then apply it to $K$-types and the Dirac index of $A_\mathfrak{q}(\lambda)$.}

\endabstract



\maketitle


\newcommand{\td}{\mathrm{d}}
\newcommand{\C}{\mathrm{C}}
\newcommand{\e}{\mathrm{e}}
\newcommand{\id}{\mathrm{id}}
\newcommand{\coker}{\mathrm{coker}}
\newcommand{\im}{\mathrm{im~}}
\newcommand{\Hom}{\mathrm{Hom}}

\newcommand{\FF}{\mathscr{F}}
\newcommand{\GG}{\mathscr{G}}
\newcommand{\oo}{\mathcal{O}}

\newcommand{\Z}{\mathbb{Z}}
\newcommand{\proj}{\mathbb{P}}
\newcommand{\cpl}{\mathbb{C}}
\newcommand{\re}{\mathbb{R}}

\newcommand{\selfb}[1]{\noindent\fbox{\parbox{\textwidth}{#1}}}

\newcommand{\homo}{\tilde{H}}
\newcommand{\homd}{H^\Delta}
\newcommand{\homcw}{H^{\mathrm{CW}}}

\newcommand{\norm}[1]{\Vert #1 \Vert}

\newcommand{\fk}{\mathfrak{k}}
\newcommand{\fg}{\mathfrak{g}}
\newcommand{\ft}{\mathfrak{t}}
\newcommand{\fs}{\mathfrak{s}}
\newcommand{\fq}{\mathfrak{q}}
\newcommand{\fl}{\mathfrak{l}}
\newcommand{\fu}{\mathfrak{u}}

\newcommand{\normlambda}[1]{\Vert #1 \Vert_{\rm lambda}}
\newcommand{\normspin}[1]{\Vert #1 \Vert_{\spin}}

\numberwithin{equation}{section} 

\section{Introduction}

Let $G$ be a real reductive group. The cohomologically induced modules $A_\fq(\lambda)$ exhaust the unitary representations of $G$ when the infinitesimal character is real, integral and strongly regular \cite{Sa}. At singular infinitesimal characters the situation is much less understood. In the weakly fair range (see Definition \ref{def.positivity}), the $A_\fq(\lambda)$ modules still provide a long list of unitary modules. For the group $G=U(p,q)$, Vogan conjectured that the modules $A_\fq(\lambda)$ include all the unitary modules with integral infinitesimal characters:

\begin{conjecture}[Vogan]\cite{Trapa2001}\label{conj.vogan.begin}
	The cohomologically induced modules $A_\fq(\lambda)$ in the weakly fair range exhaust the unitary Harish-Chandra modules for $U(p,q)$ whose infinitesimal character is a weight-translate of $\rho$.
\end{conjecture}

In the case of $U(p,q)$, an $A_\fq(\lambda)$ in the weakly fair range is either irreducible or zero, and this fact still holds up to the mediocre range (see Definition \ref{def.positivity}). Therefore, a natural problem is to establish a non-vanishing criterion for $A_\fq(\lambda)$. Trapa established such a criterion in \cite[Section 7]{Trapa2001} when studying the algorithm of the annihilators of $A_\fq(\lambda)$ using $\nu$-quasitableaux (see Definition \ref{def.anti-tableau}). Here we briefly state the criterion. He first constructed a  $\nu$-quasitableaux for each $A_\mathfrak{q}(\lambda)$, and divided each $\nu$-quasitableau into unions of difference-one skew columns (see Definition \ref{def.skew.column}). Then he introduced two new definitions, overlap and singularity, for each pair of adjacent difference-one skew columns. Suppose the $A_\mathfrak{\fq}(\lambda)$ is nonzero, then for each pair of adjacent difference-one skew columns, the singularity should be no larger than the overlap. However, the definition of the overlap involves heavy calculations on Young tableaux. This fact makes the criterion not efficient in applications.

In this paper, we prove a formula for the overlap. The formula simply adds up some integers which are given by the $\theta$-stable parabolic subalgebra $\fq$. We then render the criterion in \cite{Trapa2001} operable for applications. In particular, this criterion becomes pretty simple for $A_\fq(\lambda)$ in the nice range which is a smaller range than the weakly fair range, see Theorem \ref{main.result}. We will apply this criterion to the study of $K$-types and the Dirac index of $A_\fq(\lambda)$. 

It is worth noting that Conjecture \ref{conj.vogan.begin} is proved in the case of $U(p,2)$ in \cite{WZ24}. Hopefully, our result in this paper can help prove the conjecture for general $U(p,q)$ in the future.

This paper is organized as follows. Section 2 will provide some preliminaries. In section 3, we will prove a formula for the overlap. In section 4, we update the non-vanishing criterion for $A_\fq(\lambda)$ and apply it to the $K$-types. We prove that $K$-types of a particular highest weight must exist in $A_\fq(\lambda)$. Section 5 applies the criterion to the Dirac index. We give a non-vanishing criterion for the Dirac index of $A_\fq(\lambda)$ in the nice range.

\section{Preliminaries}

Let $G=U(p,q)$ be the group of complex linear transformations of $\cpl^{p+q}$ preserving a Hermitian form defined by a diagonal matrix $I_{p,q}$ with $p$ pluses and $q$ minuses on the diagonal. The way we arrange those signs in $I_{p,q}$ will be precisely described in the later subsection. Let $K\cong U(p)\times U(q)$ be the fixed points of the Cartan involution of inverse conjugate transpose. On the Lie algebra level, let $\theta$ be the differentiated Cartan involution and let $\fg_0=\fk_0+\fs_0$ be the corresponding Cartan decomposition. We drop the subscript for complexification.

\subsection{$K$-conjugacy classes of $\theta$-stable parabolic subalgebras of $\fg$}

We will need a very explicit description of $K$-conjugacy classes of $\theta$-stable parabolic subalgebras $\fq=\fl\oplus\fu$ of $\fg$. Let $\{(p_1,q_1),\cdots,(p_r,q_r)\}$ be an ordered sequence of pairs of non-negative integers and we assume each pair is not $(0,0)$. Set $p=\sum_i p_i$, $q=\sum_iq_i$. The matrix $I_{p,q}$ is defined in the following way: $I_{p,q}$ consists of $p_1$ pluses, then $q_1$ minuses, then $p_2$ pluses, and so on. Let $E_{ij}$ be the matrix where the $(i,j)$-th entry is $1$ and the rest entries are zero. Then $E_{ij}$ is in $\fk$ if and only if both the $i$-th and the $j$-th diagonal entry of $I_{p,q}$ have the same sign.

Fix the diagonal torus $T\subset K$ with Lie algebra $\ft_0$, and set $\ft_\re=\sqrt{-1}\ft_0$. Write $\Delta(\fg,\ft)$ for the roots of $\ft$ in $\fg$ and make standard choice of positive roots, $\Delta^+(\fg,\ft)=\{ e_i-e_j|i<j\}$. Let $\rho$ denote the half-sum of all positive roots. Then the simple roots of $\mathfrak g$ with respect to this choice of positive roots are $\alpha_i:=e_i-e_{i+1}$. And, a weight $\nu=(v_1,\dots,v_n)\in\ft_\re^\ast$ is $\fg$-dominant if and only if $\nu_1\geqslant \nu_2\geqslant\cdots\geqslant\nu_n$. 

Denote $\Delta(\fk,\ft)$ as the root system of $\fk$. Then a root $e_i-e_j$ is in $\Delta(\fk,\ft)$ if and only if the $i$-th and the $j$-th diagonal entry of $I_{p,q}$ have the same sign. We choose $\Delta^+(\fk,\ft)=\Delta^+(\fg,\ft)\cap \Delta(\fk,\ft)$. Then a weight $\nu=(v_1,\dots,v_n)\in\ft_\re^\ast$ is $\mathfrak k$-dominant if and only if $\nu_i\geqslant \nu_j$ for all $i,j$ such that $e_i-e_j\in \Delta^+(\fk,\ft)$.

Let $\fl$ denote the block diagonal subalgebra $\mathfrak{gl}(n_1,\cpl)\oplus\cdots\oplus\mathfrak{gl}(n_r,\cpl)$, and let $\fu$ denote the strict block upper-triangular subalgebra with respect to $\fl$. Write $\fq=\fl\oplus\fu$. Then $\fq$ is a $\theta$-stable parabolic subalgebra of $\fg$. As the ordered sequences range over all ordered pairs non-negative integers $\{(p_1,q_1),\cdots,(p_r,q_r)\}$ such that $p=\sum_i p_i, q=\sum_i q_i$, all the $\fq$ constructed in this way exhaust the $K$-conjugacy classes of $\theta$-stable parabolic subalgebras of $\fg$, see \cite[Example 4.5]{V97}. A $\theta$-stable parabolic $\fq=\fl\oplus\fu$ \textbf{attached to} $\{(p_1,q_1),\cdots,(p_r,q_r)\}$ means the one described above. 

\subsection{The module $A_\mathfrak{q}(\lambda)$ and ranges of positivity}

Let $\mathfrak q$ be a $\theta$-stable parabolic subalgebra attached to a sequence $\{(p_1,q_1),\cdots,(p_r,q_r)\}$, and write $n_i=p_i+q_i$. Let $L$ be the Levi subgroup of $\fq$. Any unitary 1-dimensional $(\fl,L\cap K)$-module $\cpl_\lambda$, restricted to $T$, has differential
\begin{equation}\label{lambda.coor}
	\lambda=(
	\overbrace{\lambda_1,\cdots,\lambda_1}^{n_1}|
	\cdots\cdots|
	\overbrace{\lambda_r,\cdots\lambda_r}^{n_r}
	)\in\ft_\re^\ast,
\end{equation}
with each $\lambda_i\in\Z$. By definition, the module $A_\fq(\lambda)$ is the cohomologically induced module $\mathcal L_{\dim \fu\cap\fk}(\cpl_\lambda)$, where
$$\mathcal L_j(\cpl_\lambda)=(\Pi_{\fg,L\cap K}^{\fg,K})_j\big(\ind_{\bar\fq,L\cap K}^{\fg,L\cap K}(\cpl_\lambda\otimes_\cpl\wedge^\text{top} \mathfrak{u})\big).$$
Here $\Pi_j$ is the derived Bernstein functor. Note that the infinitesimal character of $A_\mathfrak{q}(\lambda)$ is $\lambda+\rho$. See more detail of cohomological induction in \cite{KnappVoganBook}.

In the introduction, we mentioned some ranges of positivity for $\lambda$ and $\mathfrak q$. Here we list their definitions.

\begin{defi}\label{def.positivity}\cite{KnappVoganBook}
	The $\lambda$ is said to be in the \textbf{weakly good} (resp. \textbf{good}) range if $\lambda+\rho$ is $\fg$-dominant (resp. strongly $\fg$-dominant).  The $\lambda$ is said to be in the \textbf{weakly fair} (resp. \textbf{fair}) range if $\lambda+\rho(\fu)$ is $\fg$-dominant (resp. strongly $\fg$-dominant). The character $\lambda$ is said to be in the \textbf{mediocre} range for $\fq=\fl\oplus\fu$ if $\ind^\fg_{\bar{\fq}}(\cpl_{\lambda+t\rho(\fu)}\otimes \bigwedge^{\text{top}}\fu )$ is irreducible for all $t\geqslant 0$. 
\end{defi}

The next lemma makes these ranges explicit in the $U(p,q)$ setting.

\begin{lemma}\label{RangeOfPositivityLambda}\cite{Trapa2001}
	Suppose $G=U(p,q)$. Let $\fq=\fl\oplus\fu$ be attached to $\{(p_1,q_1),\cdots,(p_r,q_r)\}$, and let $\cpl_\lambda$ be a 1-dimensional $(\fl,L\cap K)$-module. Write
	\begin{equation}\label{Lambda+RhoCoor}
		\lambda+\rho=\nu=(
		\nu_1^{(1)},\cdots,\nu_{n_1}^{(1)};
		\cdots\cdots;
		\nu_1^{(r)},\cdots,\nu_{n_r}^{(r)}
		)
	\end{equation}
	where $n_i=p_i+q_i$. Then the definition of (weakly) good, (weakly) fair, and mediocre can be equivalently rephrased in the following way.
	\begin{enumerate}
\item The $\lambda$ is in the weakly good (resp. good) range if, for each pair $i<j$, $\nu^{(i)}_{n_i}\geqslant\nu^{(j)}_1$, (resp. $\nu^{(i)}_{n_i}>\nu^{(j)}_1$);
\item The $\lambda$ is in the weakly fair (resp. fair) range if, for each pair $i<j$, $\nu^{(i)}_1+\nu^{(i)}_{n_i}\geqslant\nu^{(j)}_1+\nu^{(j)}_{n_j}$ (resp.$\nu^{(i)}_1+\nu^{(i)}_{n_i}>\nu^{(j)}_1+\nu^{(j)}_{n_j}$);
\item The $\lambda$ is in the mediocre range if, for each pair $i<j$, either $\nu^{(i)}_1\geqslant \nu^{(j)}_1$ or $\nu^{(i)}_{n_i}\geqslant \nu^{(j)}_{n_j}$.
	\end{enumerate}
\end{lemma}

From the above lemma, one can easily see that the mediocre range contains the fair range, and the fair range contains the good range. This fact is true not only for $U(p,q)$, but also for general cases (\cite[Theorem 5.105]{KnappVoganBook}). These ranges are related to the singularity of the infinitesimal character of $A_\mathfrak{q}(\lambda)$. When $\lambda$ is in the good range, there is no singularity in the infinitesimal character. And the singularity accumulates starting from the weakly good range.

When discussing the $A_\mathfrak{q}(\lambda)$ module of $U(p,q)$, we introduce one more range of positivity, which is called \emph{nice}. Trapa originally defined the \emph{nice position} in the context of Young tableaux in \cite[Section 6]{Trapa2001}. Here we extend the notion to $A_\mathfrak{q}(\lambda)$. We will see the importance of the nice range in Section 4.

\begin{defi}\label{Def-Nice}
	Let $\fq$, $\lambda$ and $\nu$ be the same as in Lemma \ref{RangeOfPositivityLambda}. We say $\lambda$ is in the \textbf{nice} range for $\fq$ if, for each pair $i<j$, we have both
	$\nu^{(i)}_1\geqslant \nu^{(j)}_1$ and $\nu^{(i)}_{n_i}\geqslant \nu^{(j)}_{n_j}$.
\end{defi}

It is not hard to see that
\begin{center}
	weakly good $\Rightarrow$ nice $\Rightarrow$ weakly fair $\Rightarrow$ mediocre.
\end{center}

A module $A_\mathfrak{q}(\lambda)$ is said to be good, or in the good range if $\lambda$ is in the good range for $\mathfrak q$. Similar terminology applies for weakly good, nice, fair, weakly fair, and mediocre. When $\lambda$ is in the mediocre range, $A_\fq(\lambda)$ is either irreducible unitary or zero. This fact is correct for $U(p,q)$ but not for general reductive groups. Chapter 8 of \cite{KnappVoganBook} provides sufficient conditions from which to conclude this fact.

\subsection{Young Tableaux and $A_q(\lambda)$ modules of $U(p,q)$}

We first show the definition of $\nu$-tableaux and $(p,q)$-signed tableaux.

\begin{defi}\label{def.young.diag}
Given a partition of a positive integer $k=k_1+k_2+\cdots+k_s$ with $k_i$ decreasing, we can associate a left-justified arrangement of $k$ boxes, where the $i$-th row contains $k_i$ boxes. Such an arrangement is called a \textbf{Young diagram} of size $k$.
\end{defi}

\begin{defi}\label{def.anti-tableau}
	If $\nu=(\nu_1,\cdots,\nu_k)$ is a $k$-tuple of real numbers, a \textbf{$\nu$-quasitableau} is defined to be any arrangement of $\nu_1,\cdots,\nu_k$ in a Young diagram of size $k$. If a $\nu$-quasitableau satisfies the condition that the entries weakly increase across rows and strictly increase down columns, it is called a $\nu$-tableau. Replacing "increase" by "decrease" in the definition of $\nu$-tableau defines a \textbf{$\nu$-antitableau}. The underlying Young diagram of a quasitableau/antitableau is called its shape.
\end{defi}

\begin{defi}\label{def.signed.tableau}
	A \textbf{$(p,q)$-signed tableau} is an equivalence class of Young diagrams whose boxes are filled with $p$ pluses and $q$ minuses so that the signs alternate across rows; two signed Young diagrams are equivalent if they can be made to coincide by interchanging rows of equal length. For writing convenience, we may drop $(p,q)$ and just call it a signed tableau if there is no ambiguity.
\end{defi}

An $A_\fq(\lambda)$ module of $U(p,q)$ can be determined by a pair consisting of its annihilator and its asymptotic support. By Joseph's parameterization of primitive ideals of $\mathfrak{gl}(n,\cpl)$, an annihilator can be represented by a $\nu$-antitableau. And by \cite[Theorem 9.3.3]{CMc}, an asymptotic support can be represented by a $(p,q)$-signed tableaux. In sections 4 and 5 of \cite{Trapa2001}, there is detail about the annihilators and asymptotic supports.

\begin{thm}\label{BarbaschVogan}\cite{BarbaschVogan}
	Let $\rho+\Z^n$ be a weight lattice translate of the infinitesimal character of the trivial representation. Suppose $\nu\in \rho+\Z^n$. The map assigning an irreducible Harish-Chandra module for $U(p,q)$ with infinitesimal character $\nu$ to the pair consisting of its annihilator and its asymptotic support is an injection. On the level of tableaux, the map assigns a $\nu$-antitableau and a $(p,q)$-signed tableau of the same shape to each irreducible module with infinitesimal character $\nu$, and any such pair arises in this way. 
\end{thm}

Given an $A_\mathfrak{q}(\lambda)$ module of $U(p,q)$, the following algorithm shows how to determing its $(p,q)$-signed tableau. Note that the signed tableau is independent of the parameter $\lambda$.

\begin{lemma}\label{sign.tab.for.fq}\cite[Lemma 5.6]{Trapa2001}
Suppose $\fq$ is attached to $\{ (p_1,q_1),\cdots,(p_r,q_r) \}$ and let $p=\sum p_i$, $q=\sum q_i$. Write $n_i=p_i+q_i$. The $(p,q)$-signed tableau $S_\pm$ is constructed as follows. 

For the first $p_1$ pluses and $q_1$ minuses, put all of them in the first column. There are $n_1$ rows and each row has only one sign. The order does not matter. If $r=1$, then the construction is done.

If $r>1$, the construction is continued by induction as follows. In the $i$-th step, we obtain a signed tableau, denoted by $S^i_\pm$, by adding $p_i$ pluses and $q_i$ minuses to the tableau $S^{i-1}_\pm$ constructed in the previous step. Now, suppose we already have $S^{i-1}_\pm$ and plan to construct $S^{i}_\pm$. We should follow these rules:
	\begin{enumerate}
		\item At most one sign is added to each row-end; add signs from the top to bottom, until the new signs run out. In total there are $p_{i}$ pluses and $q_{i}$ minuses.
		\item If a row ends with a plus, add a minus to it; if a row ends with a minus, add a plus to it. Equivalently, in each row the signs are arranged alternatively.\label{rule.build.sign.tab.alter.signs}
		\item Start the arrangement at the first row of $S^{i}_\pm$. If pluses or minuses run out when doing (\ref{rule.build.sign.tab.alter.signs}), skip the row. Otherwise, do not skip.
		\item If all rows of $S^{i}_\pm$ have new signs added and there are still more signs remaining, start a new row for each sign left over.
		\item After all $n_i$ signs are arranged, the resulting diagram may not necessarily have rows of decreasing length, but one can choose a tableau equivalent to $S^{i-1}_\pm$ and do the process again so that the result does have rows of decreasing length.
	\end{enumerate}
Eventually, any tableaux we obtain must be a representative of the same signed tableau $S_\pm$. We say that this $S_\pm$ \textbf{is attached to $\fq$}.
\end{lemma}
\begin{example}\label{example.sign.tab}
Let $\fq$ be a $\theta$-stable parabolic subalgebra attached to $\{(2,1),(3,1),(0,2)\}$. The signed tableau attached to $\fq$ is constructed in the following picture. Notice that from step 1 to step 2, we need to switch the second row and the third row due to rule (5) of Lemma \ref{sign.tab.for.fq}. Similarly, we need to switch the first row and the second row from step 2 to step 3. In step 3, we skip the second row when adding minus signs as mentioned in rule (3) of Lemma \ref{sign.tab.for.fq}. 
\begin{center}\begin{tikzpicture}
		\foreach \x in {0,-0.5,-1}
		\draw (0,\x) rectangle (0.5,\x-0.5);
		\node at (0.25,-0.25) {$+$};
		\node at (0.25,-0.75) {$+$};
		\node at (0.25,-1.25) {$-$};
		
		\foreach \y in {1,4}
		\draw[->] (\y,-1) -- (\y+1,-1);
		
		\foreach \x in {0,-0.5,-1,-1.5,-2}
		\draw (2.5,\x) rectangle (3,\x-0.5);
		\foreach \x in {0,-0.5}
		\draw (3,\x) rectangle (3.5,\x-0.5);
		
		\node at (2.75,-0.25) {$+$}; \node at (3.25,-0.25) {$-$};
		\node at (2.75,-0.75) {$-$}; \node at (3.25,-0.75) {$+$};
		\node at (2.75,-1.25) {$+$};
		\node at (2.75,-1.75) {$+$};
		\node at (2.75,-2.25) {$+$};
		
		\foreach \x in {0,-0.5,...,-2}
		\draw (5.5,\x) rectangle (6,\x-0.5);
		\foreach \x in {0,-0.5,-1}
		\draw (6,\x) rectangle (6.5,\x-0.5);
		\draw (6.5,0) rectangle (7,-0.5);
		
		\node at (5.75,-0.25) {$-$}; \node at (6.25,-0.25) {$+$};
		\node at (5.75,-0.75) {$+$}; \node at (6.25,-0.75) {$-$};
		\node at (5.75,-1.25) {$+$};
		\node at (5.75,-1.75) {$+$};
		\node at (5.75,-2.25) {$+$};
		
		\node at (6.75,-0.25) {$-$};
		\node at (6.25,-1.25) {$-$};
\end{tikzpicture}\end{center}
\end{example}

After the signed tableau of $A_\fq(\lambda)$ is obtained, one can then easily construct the $(\lambda+\rho)$-antitableau as stated in the following theorem, which works for $A_\fq(\lambda)$ in the good range. Later in section 4, we will prove that the following theorem can be extended to the nice range.

\begin{thm}\label{AqLambdaToTAbleau}\cite[Theorem 6.4]{Trapa2001}
	Let $\fq=\fl\oplus\fu$ be attached to $\{(p_1,q_1),\cdots,(p_r,q_r)\}$, and let $\cpl_\lambda$ be a 1-dimensional $(\fl,L\cap K)$-module in the good range for $\fq$. Write $n_i=p_i+q_i$ and
	\begin{equation}
\lambda+\rho=\nu=(
\nu_1^{(1)},\cdots,\nu_{n_1}^{(1)};
\cdots\cdots;
\nu_1^{(r)},\cdots,\nu_{n_r}^{(r)}
)\in\ft_\re^\ast.
	\end{equation}
The tableau parameters (Theorem \ref{BarbaschVogan}) of $A_\fq(\lambda)$ are obtained inductively as follows. Start with the empty pair of tableaux and assume that the $(s-1)$-th step has been completed giving a pair $(S^{s-1},S^{s-1}_\pm)$. $S_\pm^{s}$ is obtained by adding $p_s$ pluses and $q_s$ minuses to $S_\pm^{s-1}$ according to the algorithm in Lemma \ref{sign.tab.for.fq}; $S^{s}$ is the tableau of the same shape as $S_\pm^{s}$ obtained by adding the coordinates $\nu_1^{(s)},\cdots,\nu_{n_s}^{(s)}$ sequentially from top to bottom in the remaining unspecified boxes.
\end{thm}

\begin{defi}\cite{Trapa2001}\label{def.skew.column}
	A skew diagram is any diagram obtained by removing a smaller Young diagram from a larger one that contains it. A \textbf{skew column} is a skew diagram whose shape consists of at most one box per row. A skew column in a $\nu$-quasitableau is called \textbf{difference-one} if its consecutive entries decrease by exactly one when moving down the column.
\end{defi}

\begin{defi}\label{def.partition}
	Let $\{S_1,\cdots,S_r\}$ be a set of disjoint skew columns of a Young diagram $S$, and suppose $S=\coprod S_i$. Then we say that the set $\{S_i\}$ \textbf{forms a partition of $S$} into skew columns if $S^j:=\coprod_{i\leqslant j}S_i$ is a Young diagram for each $j=1,2,\cdots,r$.
\end{defi}

\noindent\emph{Remarks.} (1) For an $A_\mathfrak{q}(\lambda)$ in the good range, Theorem \ref{AqLambdaToTAbleau} gives a $\nu$-antitableau. If $A_\mathfrak{q}(\lambda)$ is in the mediocre range, the same process can still give a $\nu$-quasitableau but not necessarily a $\nu$-antitableau.

(2) Theorem \ref{AqLambdaToTAbleau} attaches the $\nu$-antitableau a partition into difference-one skew columns, denoted as $S=\coprod S_i$. In fact, this partition is essentially given by Lemma \ref{sign.tab.for.fq} where each $S_{\pm,i}:=S^i_\pm \backslash S^{i-1}_\pm$ is a skew column. Since the $\nu$-quasitableau and the $(p,q)$-signed tableau have the same shape, the $\nu$-quasitableau naturally inherits the partition from the $(p,q)$-signed tableau. A direct corollary is that this partition into skew columns is completely determined by $\mathfrak q$.

\begin{defi}
	Let $A_\mathfrak{q}(\lambda)$ be in the mediocre range. Let $S$ be the $\nu$-quasitableau (which could be a $\nu$-antitableau) constructed in Theorem \ref{AqLambdaToTAbleau}. We call $S$ the $\nu$-quasitableau or $\nu$-antitableau \textbf{attached to $A_\mathfrak{q}(\lambda)$}.
	
	Let $S_i$ be the skew columns obtained in the $i$-th step in Theorem \ref{AqLambdaToTAbleau}. We call $S=\coprod S_i$ the \textbf{$\mathfrak{q}$-consistent partition}.
\end{defi}

\begin{example}\label{example.diff.one.skew.col}
Take the setting of Example \ref{example.sign.tab}. Let $\lambda=(0,0,0,2,2,2,2,4,4)$. Then $\lambda+\rho=(4,3,2,3,2,1,0,1,0)$. The $\nu$-quasitableau attached to this $A_\mathfrak{q}(\lambda)$ is the last tableau in the following picture. Although the $\lambda$ is not in the good range for $\mathfrak q$, the $\nu$-quasitableau is still a $\nu$-antitableau. The $\mathfrak q$-consistent partition is shown in the picture left to the $\nu$-antitableau.

If we take another $\lambda^\prime=(0,0,0,3,3,3,3,4,4)$, then we will not have a $\nu$-antitableau but only a $\nu$-quasitableau.

\begin{center}\begin{tikzpicture}		
\foreach \x in {0,-0.5,...,-2}
\draw (-1.5,\x) rectangle (-1,\x-0.5);
\foreach \x in {0,-0.5,-1}
\draw (-1,\x) rectangle (-0.5,\x-0.5);
\draw (-0.5,0) rectangle (0,-0.5);

\node at (2.75-4,-0.25) {$4$};
\node at (2.75-4,-0.75) {$3$};
\node at (2.75-4,-1.25) {$2$};

\foreach \x in {0,-0.5,...,-2}
\draw (0.5,\x) rectangle (1,\x-0.5);
\foreach \x in {0,-0.5,-1}
\draw (1,\x) rectangle (1.5,\x-0.5);
\draw (1.5,0) rectangle (2,-0.5);

\node at (1.25,-0.25) {$3$};
\node at (1.25,-0.75) {$2$};
\node at (0.75,-1.75) {$1$};
\node at (0.75,-2.25) {$0$};

\foreach \x in {0,-0.5,...,-2}
\draw (2.5,\x) rectangle (3,\x-0.5);
\foreach \x in {0,-0.5,-1}
\draw (3,\x) rectangle (3.5,\x-0.5);
\draw (3.5,0) rectangle (4,-0.5);

\node at (3.75,-0.25) {$1$};
\node at (3.25,-1.25) {$0$};

\foreach \x in {0,-0.5,...,-2}
		\draw (5.5,\x) rectangle (6,\x-0.5);
		\foreach \x in {0,-0.5,-1}
		\draw (6,\x) rectangle (6.5,\x-0.5);
		\draw (6.5,0) rectangle (7,-0.5);
		
		\node at (5.75,-0.25) {$4$}; \node at (6.25,-0.25) {$3$};\node at (6.75,-0.25) {$1$};
		\node at (5.75,-0.75) {$3$}; \node at (6.25,-0.75) {$2$};
		\node at (5.75,-1.25) {$2$}; \node at (6.25,-1.25) {$0$};
		\node at (5.75,-1.75) {$1$};
		\node at (5.75,-2.25) {$0$};

\end{tikzpicture}\end{center}
\end{example}

In \cite[section 6]{Trapa2001}, the ranges of positivity of Lemma \ref{RangeOfPositivityLambda} is translated to the level of tableaux.
\begin{defi}\label{tab.position}
Let $S$ be a $\nu$-quasitableau and $S=\coprod S_i$ be a partition of $S$ into difference-one skew columns. 
	\begin{enumerate}
		\item $S_i$ and $S_j$ are said to be in \textbf{(weakly) good} position if the smallest entry in $S_i$ is (weakly) greater than the largest entry in $S_j$.
	\item $S_i$ and $S_j$ are said to be in \textbf{nice} position if both the smallest entry in $S_i$ is greater than or equal to the smallest entry of $S_j$, and the largest entry in $S_i$ is greater than or equal to the largest entry in $S_j$.
		\item $S_i$ and $S_j$ are said to be in \textbf{(weakly) fair} position if the average of the entries in $S_i$ is (weakly) greater than the average of the entries in $S_j$.
	\item $S_i$ and $S_j$ are said to be in \textbf{mediocre} position if either one of the following conditions holds: the smallest entry in $S_i$ is greater than or equal to the smallest entry in $S_j$; or the largest entry in $S_i$ is greater than or equal to the largest entry of $S_j$.
	\end{enumerate}
The entire partition is called (weakly) good, nice, (weakly) fair, and mediocre if all pairs of its skew columns are in the specified position.
\end{defi}

\begin{lemma}\label{lem.ranges}
Let $S$ be the $\nu$-quasitableau attached to $A_\mathfrak{q}(\lambda)$, and $S=\coprod S_i$ be the $\mathfrak q$-consistent partition into difference-one skew columns. Then the partition is in (weakly) good, nice, (weakly) fair or mediocre position if and only if the $A_\mathfrak{q}(\lambda)$ is in the (weakly) good, nice, (weakly) fair or mediocre range, respectively.
\begin{proof}
	This is obvious by Definition \ref{tab.position} and \ref{Def-Nice}, and Lemma \ref{RangeOfPositivityLambda}.
\end{proof}
\end{lemma}

Let $S$ be the $\nu$-quasitableau attached to an $A_\mathfrak{q}(\lambda)$ in the mediocre range, and assume $S$ is not a $\nu$-antitableau. Then $S$ cannot represent the annihilator of the $A_\mathfrak{q}(\lambda)$ by Joseph's parameterization of primitive ideals of $\mathfrak{gl}(n,\cpl)$. Trapa fixed this issue in \cite{Trapa2001} by developing an algorithm on the $\nu$-quasitableaux. The notions of overlap and singularity for tableaux are then introduced with this algorithm. Here we will briefly state Trapa's result in Theorem \ref{Trapa.main.thm}, but skip the algorithm.

\begin{defi}\label{OverlapCD}\cite{Trapa2001}
	Let $S=\coprod S_i$ be a partition into skew columns. Given two adjacent skew columns $C=S_j$ and $D=S_{j+1}$, label the entries of $C$ and $D$ (moving sequentially down each skew column) as $c_1,\cdots,c_k$, and $d_1,\cdots,d_l$. For $1\leqslant m\leqslant \min\{k,l\}$, define
\begin{center}
	condition-$m$: $c_{k-m+i}$ is strictly left of $d_i$ in $S$, for $1\leqslant i\leqslant m$.
\end{center}
	Define the \textbf{overlap} of $C$ and $D$, denoted by $\overlap(C,D)$, to be the largest $m\leqslant\min\{k,l\}$ so that condition-$m$ holds. In particular, if condition-$m$ never holds, define $\overlap(C,D)$=0.
	
	Suppose $S$ is a $\nu$-quasitableau and each $S_i$ is a difference-one skew column. The \textbf{singularity} of $C$ and $D$, denoted as $\sing(C,D)$, is defined to be the number of pairs of identical entries among the $c_i$ and $d_j$.
\end{defi}

\begin{example}\label{example.singularity}
	The picture in Example \ref{example.diff.one.skew.col} gives a partition of a antitableau into difference-one skew columns. From left to right, we denote the skew columns as $S_1, S_2$ and $S_3$. Then we have
	$$
	\overlap(S_1,S_2)=2,~\overlap(S_2,S_3)=2;~\sing(S_1,S_2)=2,~\sing(S_2,S_3)=2.
	$$
\end{example}

\begin{thm}\label{Trapa.main.thm}
	Let $\fq$ be a $\theta$-stable parabolic and $\cpl_\lambda$ be a one-dimensional $(\fl,L\cap K)$-module in the mediocre range for $\fq$ and write $\nu=\lambda+\rho$. Let $S$ be the $\nu$-quasitableau attached to $A_\mathfrak{q}(\lambda)$ and let $S=\coprod S_i$ be the $\mathfrak{q}$-consistent difference-one partition, as in Theorem \ref{AqLambdaToTAbleau}. Then there is an algorithm (described in \cite{Trapa2001}) to locate a distinguished $S^\prime=\coprod S^\prime_i$ equivalent (in the sense of \cite[Definition 7.4]{Trapa2001}) to $S=\coprod S_i$ such that either $S^\prime=0$; or $S^\prime$ is a $\nu$-antitableau and $\coprod S^\prime_i$ is in the nice position with
	\begin{equation}\label{Sing<Overlap.org}
		\overlap(S^\prime_i, S^\prime_{i+1})\geqslant\sing(S^\prime_i, S^\prime_{i+1}), ~\forall ~i.
	\end{equation}
	The module $A_\fq(\lambda)$ is nonzero if and only if the latter case holds and in this case, the annihilator $\Ann(A_\fq(\lambda))=S^\prime$.
\end{thm}
\noindent
\emph{Remark.} In the intermediate steps of the algorithm mentioned above, one will obtain a sequence of $\nu$-quasitableaux, each of which should satisfy (\ref{Sing<Overlap.org}). The consequence of failing (\ref{Sing<Overlap.org}) is that the $A_\mathfrak{q}(\lambda)$ vanishes. Hence, this theorem sets up a criterion for an $A_\mathfrak{q}(\lambda)$ to be nonzero. In this note, we do not need the detail of the algorithm and the definition of the equivalence of two $\nu$-quasitableaux. For readers that are interested, please check sections 6 and 7 of \cite{Trapa2001}. What we will study further is the nonzero criterion (\ref{Sing<Overlap.org}).

\subsection{Dirac cohomology and Dirac index}

We fix a non-degenerate invariant symmetric bilinear form $B$ on $\fg$. Then $\fk$ and $\fs$ are orthogonal to each other under $B$. Fix an orthonormal basis $\{Z_1,\cdots,Z_{\dim \fs_0}\}$ of $\fs_0$ with respect to the inner product on $\fs_0$ induced by $B$. Let $U(\fg)$ be the universal enveloping algebra, and let $C(\fs)$ be the Clifford algebra. As introduced by Parthasarathy in \cite{Par1972}, the \textbf{Dirac operator} is defined as
\begin{equation*}
D:=\sum_{i=1}^{\dim\fs_0} Z_i\otimes Z_i\in U(\fg)\otimes C(\fs).
\end{equation*}

It is easy to check that $D$ does not depend on the choice of the orthonormal basis $\{Z_i\}$. Let $\spin_G$ be a spin module for the Clifford algebra $C(\fs)$. For any $(\fg,K)$-module $X$, the Dirac operator $D$ acts on $X\otimes \spin_G$, and the \textbf{Dirac cohomology} defined by Vogan is the following $\widetilde{K}$-module:
\begin{equation*}
	H_D(X):=\ker D/(\ker D\cap \im D).
\end{equation*}
Here $\widetilde K$ is the spin double cover of $K$. That is
\begin{equation*}
	\widetilde K:=\{(k,s)\in K\times \spin(\fs_0)|Ad(k)=p(s)\},
\end{equation*}
where $Ad:K\rightarrow SO(\fs_0)$ is the adjoint map, and $p:\spin(\fs_0)\rightarrow SO(\fs_0)$ is the universal covering map.

Let $\Delta(\fs,\ft)$ be the set of roots of $\fs$, and put
\begin{equation*}
	\Delta^+(\fs,\ft)=\Delta(\fs,\ft)\cap \Delta^+(\fg,\ft),\quad
	\Delta^-(\fs,\ft)=\Delta(\fs,\ft)\cap \Delta^-(\fg,\ft)
\end{equation*}
Denote $\rho_n$ to be the half sum of roots in $\Delta^+(\fs,\ft)$. We have the corresponding isotropic decomposition
\begin{center}
	$\fs=\fs^+\oplus\fs^-$, where $\fs^+=\sum_{\alpha\in\Delta^+(\fs,\ft)}\fg_\alpha$, and $\fs^-=\sum_{\alpha\in\Delta^-(\fs,\ft)}\fg_\alpha$.
\end{center}
Then
\begin{equation*}
	\spin_G=\bigwedge\fs^+\otimes \cpl_{-\rho_n}.
\end{equation*}
The spin module decomposes into two parts as
\begin{equation*}
	\spin^+_G:=\bigwedge^{\rm even}\fs^+\otimes \cpl_{-\rho_n};\quad \spin^-_G:=\bigwedge^{\rm odd}\fs^+\otimes \cpl_{-\rho_n}.
\end{equation*}
Let $X$ be any $(\fg,K)$-module, the Dirac operator $D$ interchanges $X\otimes \spin_G^+$ and $X\otimes \spin_G^-$. Thus the Dirac cohomology $H_D(X)$ decomposes into the even part and the odd part, which will be denoted by $H^+_D(X)$ and $H^-_D(X)$ respectively. The Dirac index is defined as
\begin{equation*}
\DI(X):=H^+_D(X)-H^-_D(X),
\end{equation*}
which is a virtual $\widetilde K$-module. It is obvious that if $\DI(X)$ is nonzero, then $H_D(X)$ is nonzero. However, the converse is not true.

\section{A formula for the overlap}

\subsection{New notations for tableaux}

In this section, we have to frequently describe the relative positions of the entries of signed tableaux. To avoid redundant descriptions, we introduce some new notations.

In general, we use a pair of square brackets $[\cdot]$ to represent an entry in a tableau. Let $S=\coprod S_i$ be a partition, and we assume all symbols like $S_i$ with subscripts means the $i$-th skew columns from a tableau $S$. Let $S^k=\coprod_{i\leqslant k}S_i$ be the union of its first $k$ skew columns. Note that $S^k$ is also a Young diagram. We write $[t]^{(j)}$ as the $t$-th entry, counting from top to bottom, of the $j$-th skew column $S_j$. Note that this entry may not be located in the $t$-th row of $S$ due to possible skips in the arrangement of $S_j$. Let $\hat S_\pm$ be a representative of a signed tableaux $S_\pm$. If the sign of $[t]^{(j)}$ is known, we may append the sign to it as a subscript. For example, $[t]^{(j)}_-$ is a negative entry. When only the sign information is needed, we may simplify the notation to $[+]^{(j)}$ and $[-]^{(j)}$, or just $[+]$ or $[-]$.

If an entry $[t_1]^{(j_1)}$ is below (or in the same row of) another entry $[t_2]^{(j_2)}$, we write $[t_1]^{(j_1)}>^\updownarrow [t_2]^{(j_2)}$ (or $[t_1]^{(j_1)}\geqslant^\updownarrow [t_2]^{(j_2)}$). If an entry $[t_1]^{(j_1)}$ is right to (or in the same vertical column of) another entry $[t_2]^{(j_2)}$, we say $[t_1]^{(j_1)}>^\leftrightarrow [t_2]^{(j_2)}$ (or $[t_1]^{(j_1)}\geqslant^\leftrightarrow [t_2]^{(j_2)}$). We may reverse the direction of these inequalities to indicate the opposite situations. In this setting, $(S,\geqslant^\updownarrow)$ is an ordered set with respect to the vertical position, and $(S_\pm,\geqslant^\leftrightarrow)$ is an ordered set with respect to the horizontal position.

In a representative $\hat S_\pm$ of a signed tableau, we say $[t]^{(j)}$ is \textbf{paired up} with $[s]^{(i)}$, denoted by $[t]^{(j)}\sim [s]^{(i)}$, if they are horizontally adjacent. We define some finite sets of (pairs of) entries as follows:
\begin{center}
	\begin{tabular}{ll} 
		$[\bullet]^{(i)}\sim [\bullet]^{(i+1)}$ & the set of pairs of entries from $\hat S_{\pm,i}$ and $\hat S_{\pm,i+1}$ that are paired up;\\
		$[\times]\sim [\bullet]^{(i+1)}$ & the set of entries from $\hat S_{\pm,i+1}$ that are not paired up with\\ & entries from $\hat S_{\pm,i}$;\\
		$[\bullet]^{(i)}\sim[\times]$ & the set of entries from $\hat S_{\pm,i}$ that are not paired up in the \\ & arrangement of $\hat S_{\pm,i+1}$;\\
		$[\bullet]^{(l)}<^\updownarrow[x]$ & the set of entries from $\hat S_{\pm,l}$ that are strictly above $[x]$;\\
		$[\bullet]^{(i)}_+$ (resp. $[\bullet]^{(i)}_-$) & the set of all $[+]$ (resp. $[-]$) in $\hat S^{(i)}$;
	\end{tabular}
\end{center}
In practice, we may combine or slightly adjust these notations to describe sets of entries, and one can quickly understand the meaning based on the table above. For example, $[\bullet]^{(l)}_+\leqslant^\updownarrow[x]$ means the set of positive entries from $\hat S_{\pm,l}$ that are above or in the same row as $[x]$. We will use $\#$ to denote the cardinality of a finite set. These newly defined notations can greatly reduce wordy description of tableaux. As the first application, we restate the condition-$m$ for adjacent skew columns $S_j$ and $S_{j+1}$ in Definition \ref{OverlapCD}:

\begin{center}
	condition-$m$: $[k-m+i]^{(j)}<^\leftrightarrow[i]^{(j+1)}$ in $S$, $\forall~1\leqslant i\leqslant m$,
\end{center}
where $k$ is the length of $S_j$.

\subsection{Proof of the overlap formula}

Given a $\theta$-stable parabolic subalgebra $\mathfrak q$, we prove a formula for the overlaps of skew columns from the $\mathfrak q$-consistent partition of the signed tableau $S_\pm$.

\begin{prop}\label{sign.tab.first.prop}
	Let $\fq$ be a $\theta$-stable parabolic subalgebra attached to $\{(p_1,q_1),\cdots,(p_r,q_r)\}$. Let $S_\pm$ be a signed tableau with a partition $S_\pm=\coprod S_{\pm,i}$. Then this partition is $\mathfrak q$-consistent if and only if all representatives $\hat S_\pm$ of $S_\pm$ with the induced partition $\hat S_\pm=\coprod \hat S_{\pm,i}$ satisfy the following two properties for all $j$:\\
	(1) each skew column $\hat S_{\pm,j}$ has $p_j$ $[+]$ and $q_j$ $[-]$;\\
	(2) if a row is skipped in the arrangement of $\hat S_{\pm,j}$ in $\hat S_\pm$, then the rows of $\hat S_{\pm}^j$, starting from the first skipped row, end with the same sign until the last row of $\hat S_{\pm,j}$.
\end{prop}

Example \ref{example.sign.tab} fits in this proposition very well. In the second step, row 3 is skipped; and we see until row 5 all rows end with $[+]$. In the last step, row 2 is skipped; and we see both the second row and the third row end with $[-]$.

\begin{proof}[Proof of Proposition \ref{sign.tab.first.prop}]
Suppose $S_\pm=\coprod S_i$ is a $\mathfrak q$-consistent partition. Take a representative $\hat S_\pm$ and let $\hat S_\pm=\coprod \hat S_i$ be the induced partition. Then rule (1) of Lemma \ref{sign.tab.for.fq} implies property (1). Rule (4) of Lemma \ref{sign.tab.for.fq} causes skips, and this rule is applied when one of the signs run out. Hence, part (2) of the property holds.

Suppose $S_\pm=\coprod S_{\pm,j}$ is a partition satisfying the properties. Let $\hat S_\pm$ be a representative and let $\hat S_\pm=\coprod \hat S_i$ be the induced partition. We shall prove that each skew column in this partition can be constructed following the five rules of Lemma \ref{sign.tab.for.fq}. There is no problem for the first column. For any $j>1$, suppose $\coprod_{i<j}\hat S_{\pm,i}$ is done. Consider the $j$-th skew column. Rules (1) and (4) hold because $\hat S_{\pm,j}$ is a skew-column and has $p_j$ $[+]$ and $q_j$ $[-]$. Rule (2) holds due to the definition of signed tableaux. If there is no skip, then there is nothing to prove. Suppose $[x]$ is the first skipped entry in $\hat S^{(j-1)}_\pm$. We may assume $[x]$ is negative. By property (2), all entries below $[x]$ are negative. Hence $[x]$ is the place where $[+]^{(j)}$ runs out. Rule (3) is proved. No need to check rule (5) because each $\hat S^j_\pm$ is a Young diagram.
\end{proof}

\begin{lemma}\label{Observations}
	Let $\fq$ be a $\theta$-stable parabolic subalgebra attached to $\{(p_1,q_1),\cdots,(p_r,q_r)\}$. Let $S_\pm$ be a signed tableau with a $\mathfrak q$-consistent partition $S_\pm=\coprod S_{\pm,i}$. Let $\hat S_\pm$ be a representative and let $\hat S_\pm=\coprod \hat S_{\pm,i}$ be the induced partition. Here are some observations for entries of $\hat S_\pm$.
	\begin{enumerate}
\item\label{ls1} For a fixed $i$, $t_1>t_2$ implies $[t_1]^{(i)}>^\updownarrow [t_2]^{(i)}$.
\item\label{ls4} Suppose $[t_1]$ and $[t_2]$ are the last entries of two different rows of $S^j$ for certain $j$. Then $[t_1]>^\updownarrow [t_2]$ implies $[t_1]\leqslant^\leftrightarrow [t_2]$.
\item\label{ls5} For a fixed $i$, $[s]^{(i)}\geqslant^\updownarrow [t]^{(i+1)}$ implies $[s]^{(i)}<^\leftrightarrow [t]^{(i+1)}$.
\item\label{ls6} If a $[+]$ is skipped in the arrangement of $\hat S_{\pm,j}$, then all $[-]^{(j)}$ must be strictly above this $[+]$. If a $[-]$ is skipped in the arrangement of $\hat S_{\pm,j}$, then all $[+]^{(j)}$ must be strictly above this $[-]$.
\item\label{ls7} If multiple entries $[x_k]$ are skipped in the arrangement of $\hat S_{\pm,j}$, then all $[x_k]$ have the same sign.
	\end{enumerate}
\begin{proof}
	(\ref{ls1}) is trivial. 
	
 Because the rows of a tableau have decreasing lengths, (\ref{ls4}) is true. The equal happens when the two rows have equal length.
	
Suppose $[t]^{(i+1)}$ is paired up with $[x]$ in $\hat S^{(i)}_\pm$, then $[x]\leqslant^\updownarrow [s]^{(i)}$, and $[x]<^\leftrightarrow [t]^{(i+1)}$. Notice both $[x]$ and $[s]^i$ are the row-ends of $\hat S^i_\pm$. $[x]\geqslant^\leftrightarrow[s]^{(i)}$. Hence, $[s]^{(i)}\leqslant^\leftrightarrow [x]<^\leftrightarrow [t]^{(i+1)}$. (3) is proved.
	
	(\ref{ls6}) and (\ref{ls7}) are due to Proposition \ref{sign.tab.first.prop}.
\end{proof}
\end{lemma}

The proof of the following lemma is completely combinatorial. For readers who are not interested in the detail discussion about Young diagrams, please jump to the next theorem.

\begin{lemma}\label{k-m+j.lemma.body}
Let $\fq$ be a $\theta$-stable parabolic subalgebra attached to $\{(p_1,q_1),\cdots,(p_r,q_r)\}$. Let $S_\pm$ be a signed tableau with a $\mathfrak q$-consistent partition $S_\pm=\coprod S_{\pm,i}$. Denote $m_i=\min\{p_i,q_{i+1}\}+\min\{p_{i+1},q_i\}$ and $n_i=p_i+q_i$. Then
\begin{equation}\label{k-m+j.lemma}
	[n_i-m_i+j]^{(i)}\geqslant^\updownarrow [j]^{(i+1)},~\forall~j\leqslant m_i.
\end{equation}
\begin{proof}
Fix a representative $\hat S_\pm$ and let $\hat S_\pm=\coprod \hat S_{\pm,i}$ be the induced partition.\\

\noindent\emph{Part I: The case of $p_i\geqslant q_{i+1}$ and $q_i\geqslant p_{i+1}$.} In this case, $m_i=p_{i+1}+q_{i+1}=n_{i+1}$. Define a number $s=\#([\bullet]^{(i)}<^\updownarrow [j]^{(i+1)})$. We claim $n_i-m_j+j-s> 0$. Suppose the claim is true. Then by Lemma \ref{Observations} (\ref{ls1}), we have $[n_i-m_j+j]^{(i)}>^\updownarrow[s]^{(i)}$. The definition of $s$ indicates that $[n_i-m_j+j]^{(i)}\geqslant^\updownarrow [j]^{(i+1)}$. The rest of Part I is the proof of the claim $n_i-m_j+j-s> 0$.

We compute $n_i-m_i$ and $j-s$ separately.
\begin{equation*}
	n_i-m_i=n_i-n_{i+1}=\#\big( [\bullet]^{(i)}\sim[\times] \big)-\#\big( [\times]\sim [\bullet]^{(i+1)}\big).
\end{equation*}
And,
\begin{align*}
	j-s-1
	=&\#\Big( ([\times] \sim[\bullet]^{(i+1)})\cap ([\bullet]^{(i+1)}<^\updownarrow [j]^{(i+1)}) \Big)\\
	&-\#\Big(
	([\bullet]^{(i)}\sim [\times]) \cap
	([\bullet]^{(i)}\leqslant^\updownarrow [s]^{(i)})
	\Big).
\end{align*}
Then we have
\begin{align*}
	n_i-m_i+j-s=&1+\#\big(
	([\bullet]^{(i)}\sim[\times])\cap ([\bullet]^{(i)}>^\updownarrow [s]^{(i)})
	\big)\\ \tag{$\ast$}
	&-\#\big(
	([\times]\sim[\bullet]^{(i+1)})\cap ([\bullet]^{(i+1)}\geqslant^\updownarrow[j]^{(i+1)})
	\big)\\
	=&\text{(denoted by)}~1+z_1-z_2.
\end{align*}
If $z_2=0$, then the proof is done. Assume the contrary and let
$$[d]^{(i+1)}\in ([\times]\sim[\bullet]^{(i+1)})\cap ([\bullet]^{(i+1)}\geqslant^\updownarrow[j]^{(i+1)}).$$
Notice that $[d]^{(i+1)}$ cannot be the first entry of a row due to the assumption that $p_i+q_i\geqslant p_{i+1}+q_{i+1}$. In other words, $\hat S_{\pm,i+1}$ is too short to start a new row. Thus $[d]^{(i+1)}$ must be paired up with an entry.

Without loss of generality, we may assume $[d]^{(i+1)}$ is paired up with a positive entry $[x]_+$ and hence $[d]^{(i+1)}\in [\bullet]^{(i+1)}_-$. We now prove by hypothesis that
\begin{equation*}
	\big([\bullet]^{(i)}\sim [\times]\big) \cap \big([\bullet]_+^{(i)}<^\updownarrow [d]^{(i+1)}_-\big)=\emptyset. \tag{$\ast1$}
\end{equation*}
Assume the contrary, and let $[c ]^{(i)}_+ \in \big([\bullet]_+^{(i)}<^\updownarrow [d]^{(i+1)}_-\big)$. There is a skip at $[c]^{(i)}_+$ in the arrangement of $\hat S_{\pm,i+1}$. By Proposition \ref{sign.tab.first.prop}, all entries of $\hat S_{\pm,i+1}$ below $[c]^{(i)}_+$ should be positive. This is contrary to the fact that $[d]^{(i+1)}\in [\bullet]^{(i+1)}_-$. Thus $(\ast1)$ is correct.

This $(\ast1)$ implies
\begin{equation*}
	\#\Big(([\bullet]^{(i)}\sim [\times]) \cap ([\bullet]^{(i)}>^\updownarrow [d]^{(i+1)}_-)\Big)\geqslant \#\big([\bullet]_+^{(i)}\sim [\times]\big). \tag{$\ast2$}
\end{equation*}
Because $[d]^{(i+1)}\geqslant^\updownarrow [j]^{(i+1)}\geqslant^\updownarrow [s]^{(i)}$, we have that $z_1$ is greater than the left side of $(\ast2)$. Hence
\begin{equation*}
	z_1\geqslant\#([\bullet]_+^{(i)}\sim [\times]).
	\tag{$\ast2^\prime$}
\end{equation*}
The right side of $(\ast2^\prime)$ is
$$
	\#([\bullet]_+^{(i)}\sim [\times])=p_i-q_{i+1}+\#([\times]\sim[\bullet]_-^{(i+1)}).
$$
Noting that the length of $\hat S^{i+1}_\pm$ is less than $\hat S^i_\pm$, we know all entries in $([\times]\sim[\bullet]^{(i+1)})$ cause skips. By Lemma \ref{Observations} (\ref{ls7}), they should all have the same sign. Our early assumption indicates that one of these entries, which is $[d]^{(i+1)}$, has negative sign. Thus
$$
\big([\times]\sim[\bullet]^{(i+1)}\big)\subset [\bullet]^{(i+1)}_-
$$
Continue on $(\ast2^\prime)$,
\begin{align*}
	z_1&\geqslant\#([\bullet]_+^{(i)}\sim [\times])\\
	&=p_1-q_{i+1}+\#([\times]\sim[\bullet]_-^{(i+1)})\\
	&=p_i-q_{i+1}+\#([\times]\sim [\bullet]^{(i+1)})\\
	&\geqslant p_i-q_{i+1}+z_2\geqslant z_2.
\end{align*}
In the last line, the first inequality is due to the definition of $z_2$, and the second inequality is due to the assumption $p_i\geqslant q_{i+1}$ at the beginning of Part I. Finally, we have proved $n_i-m_i+j-s=1+z_1-z_2>0$.\\

\noindent\emph{Part II: The case of $p_i\leqslant q_{i+1}$ and $q_i\leqslant p_{i+1}$.}
Our proof starts with the entry $[j]^{(i)}$ instead of $[j]^{(i+1)}$. This entry exists because $\hat S_{\pm,i+1}$ has more entries than $\hat S_{\pm,i}$. The proof is divided into two parts by assuming whether $[j]^{(i)}$ is paired up with an entry of $\hat S_{\pm,i+1}$ or not.

Suppose $[j]^{(i)}$ is paired up with an entry $[s]^{(i+1)}$ of $\hat S_{\pm,i+1}$. We claim that no entry above $[j]^{(i)}$ is skipped in the arrangement of $\hat S_{\pm,i+1}$.

Assume the claim fails, and we may assume the skipped entry $[x_1]^{(i)}$ is positive. By Lemma \ref{Observations} (\ref{ls6}, all $[\bullet]_-^{(i+1)}$ must be above this $[x_1]^{(i)}_+$. Notice that $q_{i+1}\geqslant p_i$, which means $[\bullet]_+^{(i)}$ is not enough to get all $[\bullet]_-^{(i+1)}$ paired up. Thus at least one of $[\bullet]^{(i+1)}_-$ is paired up with an entry from $\hat S_\pm^{i-1}$. Hence, an $[x_2]_+$ in $\hat S_\pm^{i-1}$ is skipped in the arrangement of $\hat S_{\pm,i}$. Then by Proposition \ref{sign.tab.first.prop}, all entries of $\hat S_{\pm,i}$ below $[x]_+$ are positive, which implies $[j]^{(i)}\in [\bullet]^{(i)}_+$. Now we a contradiction: $[s]^{(i+1)}\in [\bullet]^{(i+1)}_-$ but meanwhile all $[\bullet]_-^{(i+1)}$ should be above $[x_1]_+^{(i)}$ which is even above $[j]^{(i)}$. Thus the claim holds. 

The claim is equivalent to say that all entries above (and including) $[j]^{(i)}$ are paired up with entries from $\hat S_{\pm,i+1}$. A direct corollary is that $s\geqslant j$. Thus
$$[n_i-m_i+j]^{(i)}= [j]^{(i)}=^\updownarrow [s]^{(i+1)}\geqslant^\updownarrow[j]^{(i+1)} .$$

Suppose $[j]^{(i)}$ is not paired up with any entry from $\hat S_{\pm,i+1}$. A trivial situation is that $[j]^{(i)}$ is not skipped in the arrangement of $\hat S_{\pm,i+1}$. This means all $\hat S_{\pm,i+1}$ are above $[j]^{(i)}$. Then $[j]^{(i)}<^\updownarrow [j]^{(i+1)}$ is naturally true.

Now, we assume $[j]^{(i)}$ is skipped in the arrangement of $\hat S_{\pm,i+1}$. Without loss of generality, we may also assume $[j]^{(i)}$ is positive. We put $j=x_1$ as in the proof of the claim in (iii). The proof shows that all $[\bullet]_-^{(i+1)}$ are above $[j]^{(i)}$. Moreover, by Lemma \ref{Observations} (\ref{ls6}), all $[\bullet]_-^{(i)}$ are above $[x_2]_+$ which is above $[j]_+^{(i)}$. As a conclusion, all $[\bullet]_-^{(i)}$ and $[\bullet]_-^{(i+1)}$ are above $[j]_+^{(i)}$. By Lemma \ref{Observations} (\ref{ls7}), those $[\bullet]_-^{(i)}$ cannot be skipped. Then those positive entries paired up with $[\bullet]_-^{(i)}$ are again above $[j]_+^{(i)}$. In total,
$$
\#\big([\bullet]^{(i+1)}<^\updownarrow [j]_+^{(i)}\big)\geqslant q_i+q_{i+1}\geqslant q_i+p_i\geqslant j.
$$
Therefore, $$[n_i-m_i+j]^{(i)}= [j]^{(i)}>^\updownarrow [j]^{(i+1)}.$$

\noindent\emph{Part III: The case of $p_i<q_{i+1}$ and $q_i>p_{i+1}$.}

We first prove a fact that at least one entry of $[\bullet]^{(i)}_-$ is skipped in the arrangement of $\hat S_{\pm,i+1}$. Because $q_i>p_{i+1}$, at least one negative entry $[y_1]^{(i)}\in([\bullet]^{(i)}\sim [\times])$. For the same reason, at least one negative entry $[y_2]^{(i+1)}\in([\times] \sim [\bullet]^{(i+1)})$. Either the entry $[y_2]_-^{(i+1)}$ is the first entry of a row or it is paired up (on their left side) with an entry of $\hat S^{(i-1)}_\pm$. In the former case, $[y_2]_-^{(i+1)}>^\updownarrow [y_1]_-^{(i)}$ because $[y_2]_-^{(i+1)}$ is arranged after all $[\bullet]^{(i)}$. In the latter case, we prove $[y_2]_-^{(i+1)}>^\updownarrow [y_1]_-^{(i)}$ by hypothesis. Assume the contrary. Then in the arrangement of $\hat S_{\pm,i}$ there is a skip over $[+]$ which is the entry paired up with $[y_2]_-^{(i+1)}$. As a consequence, $[y_1]_-^{(i)}$ cannot have negative sign by Proposition \ref{sign.tab.first.prop}. Therefore, we have a contradiction. As a conclusion, $[y_2]_-^{(i+1)}>^\updownarrow [y_1]_-^{(i)}$. Then there is a skip over $[y_1]_-^{(i)}$.

Now we go back to (\ref{k-m+j.lemma}). A trivial case is that no skip happens in the arrangement of $\hat S_{\pm,i+1}$ above $[j]^{(i+1)}$. Thus $[j]^{(i+1)}$ is actually in the $j$-th row. Meanwhile, $n_i-m_i+j=q_i-p_{i+1}+j>j$, which means $[k-m+j]^{(i)}$ must be below the $j$-th row. Hence, $[k-m+j]^{(i)}>^\updownarrow[j]^{(i+1)}$.

Suppose there exists an entry $[n_0]^{(i)}<^\updownarrow [j]^{(i+1)}$ such that $[n_0]^{(i)}$ is skipped in the arrangement of $\hat S_{\pm,i+1}$. We already have an entry $[y_1]_-^{(i)}$ skipped. By Lemma \ref{Observations} (\ref{ls7}), we know $[n_0]^{(i)}$ is negative. We claim that
\begin{equation*}
	([\times]  \sim[\bullet]^{(i+1)})\cap ([\bullet]_+^{(i+1)}<^\updownarrow [j]^{(i+1)})=\emptyset.\tag{$\ast3$}
\end{equation*}
We prove the claim by hypothesis. Assume the contrary, and let $[w_1]_-$ be the entry paired up with $[w_2]^{(i+1)}_+$ which is above $[j]^{(i+1)}$. By Lemma \ref{Observations} (\ref{ls6}) and the assumption on $[n_0]^{(i)}_-$, $[w_2]^{(i+1)}_+<^\updownarrow [n_0]^{(i)}_-<^\updownarrow [j]^{(i+1)}$. Thus, there is a skip over $[w_1]_-$ in the arrangement of $\hat S_{\pm,i}$. Then by Lemma \ref{Observations} (\ref{ls6}), $[\bullet]^{(i)}_+\subset ([\bullet]^{(i)}<^\updownarrow [w_1]_-)$. At this point, we have proved that all $[\bullet]_+^{(i)}$ and $[\bullet]_+^{(i+1)}$ are above $[j]^{(i+1)}$. Therefore, $j>p_i+p_{i+1}$ which is contradict to the fact $j\leqslant m_i=p_i+p_{i+1}$. The claim is proved.

Define a number $s=\#([\bullet]^{(i)}<^\updownarrow [j]^{(i+1)})$. Then $[n_i-m_i+j]\geqslant^\updownarrow[j]^{(i+1)}$ is equivalent to $n_i-m_j+j-s> 0$. We compute the difference $n_i-m_i$ and $j-s$ separately.
\begin{align*}
	j-s-1
	=&\#\Big( ([\times] \sim[\bullet]^{(i+1)})\cap ([\bullet]^{(i+1)}<^\updownarrow [j]^{(i+1)}) \Big)\\
	&-\#\Big(
	([\bullet]^{(i)}\sim [\times]) \cap
	([\bullet]^{(i)}<^\updownarrow [j]^{(i+1)})
	\Big)\\
	=&(\text{denoted by})~z^\prime_1-z^\prime_2.
\end{align*}
And because $q_i>p_{i+1}$, we have
\begin{equation*}
q_i-p_{i+1}=\#\big([\bullet]_-^{(i)}\sim[\times]\big)-\#\big([\times]\sim[\bullet]_+^{(i+1)}\big).
\end{equation*}
Notice there is a skip over $[n_0]_-^{(i)}$. By Lemma \ref{Observations} (\ref{ls6}), $[\bullet]_+^{(i+1)}$ are all above $[n_0]_-^{(i)}$ and hence above $[j]^{(i+1)}$. Combining this fact and $(\ast3)$, we obtain that $\#\big([\times]\sim[\bullet]_+^{(i+1)}\big)=0$. By Lemma \ref{Observations} (\ref{ls7}), we have
$$
\Big(
([\bullet]^{(i)}\sim [\times]) \cap
([\bullet]^{(i)}<^\updownarrow [j]^{(i+1)})
\Big)
\subset
\big([\bullet]_-^{(i)}\sim[\times]\big).
$$
Therefore,
$$q_i-p_{i+1}\geqslant z_2^\prime.$$
Now we compute
\begin{equation*}
	n_i-m_i+j-s=q_i-p_{i+1}+z_1^\prime-z_2^\prime+1> 0.
\end{equation*}
As a result, $[n_i-m_i+j]^{(i)}\geqslant^\updownarrow [j]^{(i+1)}$.\\

\noindent\emph{Part IV: The case of $p_i>q_{i+1}$ and $q_i<p_{i+1}$.} It is exactly the same as part III.
\end{proof}
\end{lemma}

\begin{thm}\label{OverlapFormula}
Let $\mathfrak q$ be a $\theta$-stable parabolic subalgebra attached to $\{(p_1,q_1),\cdots,(p_r,q_r)\}$. Let $S=\coprod S_k$ be a $\mathfrak q$-consistent partition. Let $\hat S_\pm$ be a representative of $S_\pm$ and let $\hat S_\pm=\coprod \hat S_{\pm,k}$ be the induced partition. Then the overlap of two adjacent skew columns $S_i$ and $S_{i+1}$ is
	\begin{equation}
		\overlap(S_i,S_{i+1})
		=\min\{p_i,q_{i+1}\}+\min\{p_{i+1},q_i\}.
	\end{equation}
	\begin{proof}
Denote $m_i=\min\{p_i,q_{i+1}\}+\min\{p_{i+1},q_i\}$ and $n_i=p_i+q_i$ as in Lemma \ref{k-m+j.lemma.body}. We should prove two facts. One is that condition-$m_i$ holds, and the other one is that condition-$(m_i+1)$ fails. We have proved (\ref{k-m+j.lemma}), which is
$$
	[n_i-m_i+j]\geqslant^\updownarrow [j]^{(i+1)},~\forall~j\leqslant m_i.
$$
By Lemma \ref{Observations} (\ref{ls5}), we have
$$
	[n_i-m_i+j]<^\leftrightarrow [j]^{(i+1)},~\forall~j\leqslant m_i.
$$
Hence condition-$m_i$ holds.

In the case when $p_i\geqslant q_{i+1}$ and $q_i\geqslant p_{i+1}$, and the case when $p_i\leqslant q_{i+1}$ and $q_i\leqslant p_{i+1}$, the proof is done because $m_i$ equals the length of one of the skew columns. In the case when $q_i> p_{i+1}$ and $q_{i+1}>p_i$, we should continue to prove that condition-($m_i+1$) fails. Same for the case when 
$q_i< p_{i+1}$ and $q_{i+1}<p_i$, and the proof will be the same.

Now, suppose $q_i> p_{i+1}$ and $q_{i+1}>p_i$. We first prove a claim:
\begin{center}
	($\ast$) At least one entry of $[\bullet]^{(i+1)}_-$ is below all of $[\bullet]^{(i)}_-$.
\end{center}

Assume the contrary, then there exists one entry $[w]^{(i)}_-$ below all $[\bullet]_-^{(i+1)}$. By Proposition \ref{sign.tab.first.prop}, no positive entry above $[w]^{(i)}_-$ should be skipped in the arrangement of $\hat S_{\pm,i}$. As a result, all $[\bullet]^{(i+1)}_-$ have to be paired up with $[\bullet]^{(i)}_+$. This fact implies that $p_{i}\geqslant q_{i+1}$, which is contradict to the assumption $q_{i+1}>p_i$.

Here we prove condition-($p_i+p_{i+1}+1$) fails. Because of the claim we just proved, we let $[k_\ast]^{(i)}_-$ be the last entry of $[\bullet]_-^{(i)}$, and let $[j_\ast]_-^{(i+1)}$ be the first negative entry below $[k_\ast]^{(i)}_-$. Define
\begin{gather*}
x_1^\prime:=
	([\bullet]^{(i)} \sim [\bullet]^{(i+1)}) \cap ( [\bullet]_+^{(i+1)}<^\updownarrow [j_\ast]^{(i+1)}_- ) 
	;\\
x_2^\prime:=
	([\bullet]^{(i)} \sim [\bullet]^{(i+1)}) \cap ( [\bullet]_-^{(i)}<^\updownarrow [j_\ast]^{(i+1)}_- );\\
z_1^{\prime\prime}:=\#
([\times]\sim [\bullet]^{(i+1)}) \cap ( [\bullet]^{(i+1)}_+<^\updownarrow [j_\ast]^{(i+1)}_- );\\
z_2^{\prime\prime}:=\#
([\bullet]^{(i)}\sim [\times]) \cap ( [\bullet]_-^{(i)}<^\updownarrow [j_\ast]^{(i+1)}_- ) 
\Big).
\end{gather*}
Then
\begin{gather*}
j_\ast=x^\prime_1+x^\prime_2+z^{\prime\prime}_1+1;\\
k_\ast=x^\prime_1+x^\prime_2+z^{\prime\prime}_2.
\end{gather*}
In this case, $n_i=p_i+q_i$, and $m_i=p_i+p_{i+1}$. We compute
\begin{equation*}
	n_i-(m_i+1)+j_\ast=q_i-p_{i+1}+x^\prime_1+x^\prime_2+z^{\prime\prime}_1.
\end{equation*}
By claim ($\ast$) and Lemma \ref{Observations} (\ref{ls6}), all $[\bullet]_-^{(i)}$ and $[\bullet]_+^{(i+1)}$ are above $[j_\ast]^{(i+1)}$. Thus
\begin{equation*}
	q_i-p_i+1=z^{\prime\prime}_1-z^{\prime\prime}_2.
\end{equation*}
Eventually, we have $n_i-(m_i+1)+j_\ast=k_\ast$ and $[k-m+j_\ast]^{(i)}=[k_\ast]_-^{(i)}$. By the definition of $[k_\ast]_-^{(i)}$, we know that $[j_\ast]^{(i+1)}_->^\updownarrow [k-m+j_\ast]^{(i)}_-$. By Lemma \ref{Observations} (\ref{ls4}), $[j_\ast]^{(i+1)}_-\leqslant^\updownarrow [k-m+j_\ast]^{(i)}_-$. Hence condition-($p_i+p_{i+1}+1$) fails.
\end{proof}
\end{thm}

\section{Non-vanishing criterion for $A_\fq(\lambda)$ in the nice range}

Once we have the formula for the overlap, we can do a one-step application in Theorem \ref{Trapa.main.thm} by simply writing (\ref{Sing<Overlap.org}) as
$$
\min\{p_i,q_{i+1}\}+\min\{p_{i+1},q_i\}\geqslant \sing(S^\prime_i,S^\prime_{i+1}),
$$
where $S^\prime_i$ has $p_i$ $[+]$ and $q_i$ $[-]$ for all $i$. However, one still need to understand the algorithm mentioned in Theorem \ref{Trapa.main.thm} to figure out whether an $A_\mathfrak{q}(\lambda)$ is nonzero or not. In this section, we will prove that the situation will be much better when $A_\mathfrak{q}(\lambda)$ is in the nice range.

Using the notation in (\ref{lambda.coor}), $\lambda$ being in the nice range is equivalent to
\begin{equation}\label{lambda.nice.condition.equiv}
	\lambda_{i+1}-\lambda_i\leqslant \min\{n_i,n_{i+1} \},~\forall i.
\end{equation}

Write $\nu=\lambda+\rho$ as in (\ref{Lambda+RhoCoor}). Define
\begin{equation}
	\mathcal R_{ij}:=
	\{\nu^{(i)}_1,\cdots,\nu^{(i)}_{n_i}\}
	\cap
	\{\nu^{(j)}_1,\cdots,\nu^{(j)}_{n_j}\}.
\end{equation}
and let $R_{ij}:=\#\mathcal R_{ij}$. We write $\mathcal R_{ij}$ as $\mathcal R^{(i)}_{ij}$ when we consider it as a subset of $\{ \nu^{(i)}_i,\cdots,\nu^{(i)}_{n_i}\}$.

\begin{lemma}\label{lemma.Rij.formula}
	Let $S$ be the $\nu$-quasitableau attached to an $A_\mathfrak{q}(\lambda)$ in the nice range. Write $\lambda$ as in (\ref{lambda.coor}). Let $S=\coprod S_i$ be the $\mathfrak q$-consistent partition. Then the singularity of two adjacent difference-one skew columns is
	\begin{equation}\label{Rij.new.def}
		\sing(S_i,S_{i+1})=R_{i,i+1}=\begin{cases}
			\lambda_{i+1}-\lambda_i,\quad&\text{if}~\lambda_{i+1}>\lambda_i;\\
			0\quad& \text{if}~\lambda_{i+1}\leqslant \lambda_i.
		\end{cases}
	\end{equation}
\end{lemma}

\begin{thm}\label{main.result}
	Let $A_\fq(\lambda)$ be in the nice range, where $\fq$ is attached to $\{(p_1,q_1),\cdots,(p_r,q_r)\}$. Let $(S,S_\pm)$ be a pair of $\nu$-quasitableau and signed tableau constructed in Theorem \ref{AqLambdaToTAbleau}. Write $\lambda$ as in (\ref{lambda.coor}) and $\nu=\lambda+\rho$ as in (\ref{Lambda+RhoCoor}). Then $A_\fq(\lambda)$ is nonzero if and only if
	\begin{equation}\label{Aq.nice.condition}
		\lambda_{i+1}-\lambda_i\leqslant \min\{p_i,q_{i+1}\}+\min\{q_i,p_{i+1} \},~\forall~i
	\end{equation}
And when (\ref{Aq.nice.condition}) holds, $S$ is a $\nu$-antitableau.
\begin{proof}
Suppose $A_\fq(\lambda)$ is nonzero, then by Theorem \ref{Trapa.main.thm} and \ref{OverlapFormula}, we directly have (\ref{Aq.nice.condition}).

Now we prove (\ref{Aq.nice.condition}) implies that $S$ is a $\nu$-antitableau. The way we construct $S$ from $A_\fq(\lambda)$ indicates that the entry $[s]^{(t)}$ is $\nu_s^{(t)}$ for all $s$ and $t$. Let $\nu_j^{(i_0)}$ be an entry in the $\nu$-quasitableau. Let $x$ be the entry right and adjacent to $\nu_j^{(i_0)}$, and let $y$ be the entry below and adjacent to $\nu_j^{(i_0)}$. We shall prove that $\nu_j^{(i_0)}\geqslant x$ and $\nu_j^{(i_0)}> y$. Because $\lambda$ is in the nice range for $\fq$, $\nu^{(i_0)}_j$ either contributes to the singularity $R_{i_0,i_0+1}$ or is strictly greater than all $\{\nu^{(i_0+1)}_\bullet\}$.

\noindent\emph{Case I.} Suppose $\nu^{(i_0)}_j$ is strictly greater than all $\{\nu^{(i_0+1)}_\bullet\}$. Then $\nu^{(i_0)}_j$ is strictly greater than all $\{\nu^{(i^\prime)}_\bullet\}$ as long as $i^\prime> i_0$. Meanwhile, $\nu^{(i_0)}_j>\nu^{(i_0)}_{j^\prime}$ for all $j^\prime>j$. Notice that $x$ and $y$ are either in $\{\nu^{(i^\prime)}_\bullet| i^\prime> i_0\}$ or $\{\nu^{(i_0)}_{j^\prime}|j^\prime> j\}$. Thus $\nu_j^{(i_0)}$ is strictly greater than them.

\noindent\emph{Case II.} Suppose $\nu^{(i_0)}_j$ contributes to the singularity $R_{i_0,i_0+1}$. In other words, $\nu^{(i_0)}_j\in \mathcal R^{(i_0)}_{i_0,i_0+1}$. Suppose $\nu_{j_0}^{(i_0+1)}=\nu_j^{(i_0)}$. Because $\lambda$ is in the nice range for $\fq$, $\mathcal R^{(i_0)}_{i_0,i_0+1}$ are the bottom entries of the skew column $S_{i_0}$. Let $n_{i_0}$ be the length of $S_{i_0}$ and write $R=R_{i_0,i_0+1}$. Then $j=n_{i_0}-R+j_0$. Let $m=\overlap(S_{i_0},S_{i_0+1})$. Then by Lemma \ref{k-m+j.lemma.body}, $[n_{i_0}-m+j_0]^{(i_0)}\geqslant^\updownarrow[j_0]^{(i_0+1)}$. Meanwhile, (\ref{Aq.nice.condition}) implies $m\geqslant R$. Hence $[n_{i_0}-R+j_0]^{(i_0)}\geqslant^\updownarrow [n_{i_0}-m+j_0]^{(i_0)}$. Therefore, $[n_{i_0}-R+j_0]^{(i_0)}\geqslant^\updownarrow [j_0]^{(i_0+1)}$. And as a corollary, we have $[n_{i_0}-R+j_0]^{(i_0)}<^\leftrightarrow [j]^{(i_0+1)}$. As a conclusion, $\nu_j^{(i_0)}$ is lower than and strictly left to the entry $\nu_{j_0}^{(i_0+1)}$.

Suppose $x$ is in $\hat S_{\pm,i_x}$. If $i_x=i_0+1$, then by the conclusion in the previous paragraph, $x$ is either the entry $\nu_{j_0}^{(i_0+1)}$ or lower than $\nu_{j_0}^{(i_0+1)}$. As a result, $x\leqslant \nu_j^{(i_0)}$. If $i_x>i_0+1$, there are two situations. The trivial one is that $\nu_{j}^{(i_0)}$ is strictly greater than all of $\{\nu^{(i_x)}_\bullet\}$. In the other situation, we can consecutively apply the conclusion of the previous paragraph. Then we either end with a situation as case I, or there exists $\nu_{j_x}^{i_x}=\nu_j^{(i_0)}$. We conclude that, $\nu_{j_x}^{i_x}$ is on the right side of $\nu_j^{(i_0)}$ and no lower than $\nu_j^{(i_0)}$. Hence, $x\leqslant \nu_{j_x}^{i_x}= \nu_j^{(i_0)}$.

All the proofs for the entry $y$ can be copied here. But in this case, $y$ is strictly less than $\nu^{(i_0)}_j$ because it is below $\nu^{(i_0)}_j$. Now we have proved $S$ is a $\nu$-antitableau.

Suppose (\ref{Aq.nice.condition}) holds. Then $S$ is a $\nu$-antitableau. Theorem \ref{Trapa.main.thm} implies that this $A_\fq(\lambda)$ is nonzero.
\end{proof} 
\end{thm}

\begin{cor}
Theorem \ref{AqLambdaToTAbleau} is correct for $A_\fq(\lambda)$ in the nice range.
\end{cor}

We end this section with an application of Theorem \ref{main.result} to $K$-types.

\begin{cor}
Let $\fq$ be attached to $\{(p_1,q_1),\cdots,(p_r,q_r)\}$. Suppose an $A_\fq(\lambda)$ is nonzero and in the nice range. Then $\lambda+2\rho(\fu\cap \fs)$ must be a highest weight of a $K$-type of $A_\fq(\lambda)$.
\begin{proof}
Let
$$\lambda=(
\overbrace{\lambda_1,\cdots,\lambda_1}^{n_1}|
\cdots\cdots|
\overbrace{\lambda_r,\cdots\lambda_r}^{n_r}
)\in\ft_\re^\ast,$$
where $n_i=p_i+q_i$. For the convenience of writing, we write it as
$$
\lambda=(\prescript{(n_1)}{}\lambda_1 | \prescript{(n_2)}{}\lambda_2 |\cdots|\prescript{(n_r)}{}\lambda_r).
$$
Here $\prescript{(n)}{}\bullet$ means an $n$-tuple whose entries have the same value. The weight $2\rho(\mu\cap\fs)$ is
$$
2\rho(\mu\cap \fs)=\Big( \prescript{(p_1)}{}\eta_{1,+};\prescript{(q_2)}{}\eta_{1,-}|\prescript{(p_2)}{}\eta_{2,+};\prescript{(q_2)}{}\eta_{2,-}|\cdots|\prescript{(p_r)}{}\eta_{r,+};\prescript{(q_r)}{}\eta_{r,-}
\Big),
$$
where
$$
\eta_{j,+}=-\sum_{l<j}q_l+\sum_{t>j}q_t;\qquad \eta_{j,-}=-\sum_{l<j}p_l+\sum_{t>j}p_t.
$$
Then $\lambda+2\rho(\fu\cap\fs)$ is $\fk$-dominant if and only if
\begin{gather*}
	\lambda_{j}-\sum_{l<j}q_l+\sum_{t>j}q_t\geqslant \lambda_{j+1}-\sum_{l<j+1}q_l+\sum_{t>j+1}q_t;\\
	\lambda_{j}-\sum_{l<j}p_l+\sum_{t>j}p_t\geqslant \lambda_{j+1}-\sum_{l<j+1}p_l+\sum_{t>j+1}p_t.
\end{gather*}
After simplification, we have
\begin{equation*}
	\lambda_{j+1}-\lambda_j\leqslant p_{j+1} +p_j;\qquad	\lambda_{j+1}-\lambda_j\leqslant q_{j+1} +q_j.
\end{equation*}
Theorem \ref{main.result} directly implies the equations above. Thus $\lambda+2\rho(\fu\cap\fs)$ is $\fk$-dominant. Using the results of bottom layers from Corollary 5.85 of \cite{KnappVoganBook} and its following prose, we know that $\lambda+2\rho(\fu\cap\fs)$ is the highest weight of a $K$-type.
\end{proof}
\end{cor}

\section{Application to Dirac index}

Recently, Dong and Wong provided an equivalent condition when a weakly fair $A_\fq(\lambda)$ in the Dirac series has nonzero Dirac index in \cite{DongWong21}. On the one hand the advantage of this condition is that it provides a formula for the Dirac index for $A_\fq(\lambda)$, but on the other hand it is not easy to quickly check this equivalent condition which is based on a system of inequalities of non-negative integers. In this section, we reduce the range to the nice range, and then develop an easier-to-check non-vanishing criterion for the Dirac index of $A_\fq(\lambda)$.

\subsection{The strengthened H.P.-condition}

The Dirac index of $A_\fq(\lambda)$ is naturally zero when either $A_\fq(\lambda)$ or $H_D(A_\fq(\lambda))$ is zero. Thus we assume $A_\fq(\lambda)$ is nonzero and lives in the Dirac series. Vogan gave a conjecture about the infinitesimal characters of the $\tilde K$-types in the Dirac cohomology, and Huang and Pand\v{z}i\'{c} proved it in \cite{HP2002}. Its application, by Dong and Wong, in the case of $A_\fq(\lambda)$ of $U(p,q)$ gives the following necessary conditions for nonzero Dirac cohomology. And we call it the H.P.-condition, where H.P. stands for the authors of \cite{HP2002}.

\begin{lemma}\label{HP-condition}\cite{DongWong21}
	Assume that $H_D(A_\fq(\lambda))\neq 0$ and write $\lambda+\rho=\nu$ as
\begin{equation}\label{nu.coor}
\nu=(
	\nu_1^{(1)},\cdots,\nu_{n_1}^{(1)}|
	\cdots\cdots|
	\nu_1^{(r)},\cdots,\nu_{n_r}^{(r)}
	).
\end{equation}
	Then in the coordinate of $\nu$,
	\begin{enumerate}
		\item no entry can appear more than twice;
		\item there are at most $\min\{p,q\}$ distinct entries appearing twice.
	\end{enumerate}
\end{lemma}

They also gave a non-vanishing criterion for $\DI(A_\fq(\lambda))$ in the same paper.

\begin{lemma}\label{DiracIndexNonzero}\cite{DongWong21}
Use the notation $R_{ij}$ in section 4, and consider the inequalities for non-negative integers $\{a_{ij},b_{ij}\}$, $1\leqslant i<j\leqslant k$:
	\begin{equation}\label{DI=0.Ineq.ori}
		\begin{cases}
			a_{ij}+b_{ij}=R_{ij},\\
			\big(\sum_{x>i} a_{ix}+\sum_{y<i}b_{yi}\big)
			\leqslant p_i,\\
			\big(\sum_{x<i} a_{xi}+\sum_{y>i} b_{iy}\big)
			\leqslant q_i.
		\end{cases}
	\end{equation}
	There is a solution $\{a_{ij},b_{ij}\}$ to (\ref{DI=0.Ineq.ori}) if and only if $\DI(A_\fq(\lambda))\neq 0$. (There are five equations in \cite{DongWong21}. But two of them are redundant.)
\end{lemma}

Remark 5.8 in \cite{DongWong21} conjectured that if a weakly fair $A_\fq(\lambda)$ which satisfies the conditions in Lemma \ref{HP-condition} but (\ref{DI=0.Ineq.ori}) has no non-negative integer solution, then this $A_\fq(\lambda)$ must be the zero module. In fact, this conjecture is not correct. We find out that the problem happens at part (2) of the H.P.-condition in Lemma \ref{HP-condition}. We will modify the H.P.-condition and give a non-vanishing criterion for $\DI(A_\fq(\lambda))$ different from Lemma \ref{DiracIndexNonzero}, but only for $\lambda$ in the nice range.

For a nonzero $A_\fq(\lambda)$ in the nice range, the first part of the H.P.-condition has the following effects on the $R_{ij}$ defined in Lemma \ref{DiracIndexNonzero}.

\begin{lemma}\label{HP.condition.Rij}
	Suppose $\lambda$ only satisfies part (1) of the H.P.-condition and that $\lambda$ is in the nice range for $\fq$, where $\fq$ is attached to a sequence $\{(p_1,q_1),\cdots,(p_r,q_r)\}$. Let $\nu=\lambda+\rho$ and use the coordinate as (\ref{nu.coor}). Write $n_i=p_i+q_i$, then
	\begin{enumerate}
		\item \label{Ri.i+1.neq.0} $R_{ij}\neq0$ if and only if $j=i+1$.
		\item \label{R12+R23<n} $R_{i-1,i}+R_{i,i+1}\leqslant n_i$.
	\end{enumerate}
	\begin{proof}
		Assume the contrary of (1). Then there exist four indices $a,b,s$, and $t$ such that $s+1<t$ and $\nu_a^{(s)}=\nu_b^{(t)}$. We can pick an integer $k$ such that $s<k<t$. By the assumption of $\lambda$ being in the nice range for $\fq$, we have
		$$
		\nu^{(k)}_{n_k}\leqslant \nu^{(s)}_{n_s}\leqslant\nu_a^{(s)}=
		\nu_b^{(t)}\leqslant\nu_1^{(t)}\leqslant\nu_1^{(k)}.
		$$
		The inequality above shows that there must exist one entry $\nu^{(k)}_{x}$ equal to $\nu_a^{(s)}$ and $\nu_b^{(t)}$, which means in the coordinates of $\nu$, at least three entries are the same. This fact is contradict to part (1) of the H.P-condition. The proof of (1) is done.
		
		Claim (2) is obtained by direct computation. We may assume both of $R_{i-1,i}$ and $R_{i,i+1}$ are nonzero. Then
		$$
		R_{i-1,i}=\nu^{(i)}_1-\nu^{(i-1)}_{n_{i-1}}+1,~\text{and}~ R_{i,i+1}=\nu^{(i+1)}_1-\nu^{(i)}_{n_i}+1.
		$$
		Because $R_{i-1,i+1}=0$, we must have $\nu^{(i+1)}_1<\nu^{(i-1)}_{n_{i-1}}$. Recall that each $\lambda_i\in \Z$ and the difference of two entries of the coordinate of $\rho$ is also an integer. So, we have $\nu^{(i+1)}_1\leqslant\nu^{(i-1)}_{n_{i-1}}-1$
		Thus
		\begin{align*}
			R_{i-1,i}+R_{i,i+1}&=(\nu^{(i)}_1-\nu^{(i-1)}_{n_{i-1}}+1)+(\nu^{(i+1)}_1-\nu^{(i)}_{n_i}+1)\\
			&=(\nu^{(i)}_1-\nu^{(i)}_{n_i}+1)+(\nu^{(i+1)}_1-\nu^{(i-1)}_{n_{i-1}}+1)\\
			&=n_i+(\nu^{(i+1)}_1-\nu^{(i-1)}_{n_{i-1}}+1)\leqslant n_i.
		\end{align*}
		The proof is done.
	\end{proof}
\end{lemma}

\begin{defi}\label{strengthened.H.P}
	Let $A_\fq(\lambda)$ be in the nice range, and $\nu$, $R_{ij}$ be defined as in Lemma \ref{HP-condition}. Suppose $\fq$ is attached to the $\{(p_1,q_1),\cdots,(p_r,q_r)\}$. We call the followings the \textbf{strengthened H.P.-condition},
\begin{enumerate}
	\item No entry of coordinates of $\nu$ can appear more than twice;
	\item All $R_{i,i+1}$ should satisfy
	\begin{equation}\label{strengthened.H.P.(2)}
		\sum_{i=k}^{l-1} R_{i,i+1}\leqslant
		\min\left\{\sum_{i=k}^l p_i,\sum_{i=k}^l q_i\right\},\forall~ 1\leqslant k<l\leqslant r.
	\end{equation}
\end{enumerate}
Obviously, the second condition heavily depends on the structure of $\fq$. We say that $\lambda$ satisfies the strengthened H.P.-condition for $\fq$ if both two conditions hold.
\end{defi}

\subsection{Non-vanishing criterion for $\DI(A_\fq(\lambda))$ for $\lambda$ in the nice range}
\begin{thm}\label{DI.equiv.condition.thm}
Let $A_\fq(\lambda)$ be a nonzero module in the nice range, where $\fq$ is attached to a sequence $\{(p_1,q_1),\cdots,(p_r,q_r)\}$. Then $\DI(A_\fq(\lambda))\neq 0$ if and only if  $\lambda$ satisfies the strengthened H.P.-condition for $\fq$ (see Definition \ref{strengthened.H.P}).
\begin{proof}
When $A_\fq(\lambda)$ is in the nice range, claim (1) of Lemma \ref{HP.condition.Rij} can reduce the inequalities (\ref{DI=0.Ineq.ori}) to
\begin{equation}\label{DI=0.Ineq.Red}
	\begin{cases}
	a_{i,i+1}+b_{i,i+1}=R_{i,i+1},	&1\leqslant i\leqslant r-1, \\
	a_{i,i+1}+b_{i-1,i}\leqslant p_i, &1\leqslant i\leqslant r,	\\
	a_{i-1,i}+b_{i,i+1}\leqslant q_i, &1\leqslant i\leqslant r,
	\end{cases}
\end{equation}
where $a_{i,i+1}$, $b_{i,i+1}$ are $0$ when $i=0$ or $i=r$. And the non-vanishing criterion in Theorem \ref{main.result} can be written as
\begin{equation}\label{Sing<Overlap}
	R_{i,i+1}\leqslant \min\{p_i,q_{i+1}\}+\min\{q_i,p_{i+1}\}.
\end{equation}
In the proof, we write $n_i=p_i+q_i$.

\noindent\emph{Part I.} Assume $\DI(A_\fq(\lambda))\neq0$. A direct corollary is that $H_D(A_\fq(\lambda))\neq 0$, which implies that part (1) of the strengthened H.P.-condition holds. Let $\{a_{i,i+1},b_{i,i+1}\}_{1\leqslant i\leqslant r-1}$ be a non-negative integer solution of (\ref{DI=0.Ineq.ori}). Sum up the inequalities $b_{i-1,i}+a_{i,i+1}\leqslant p_i$ where $i$ runs from $k$ to $l$ ($k<l$). Then we have
\begin{align*}
\sum_{i=k}^l p_i & \geqslant \sum_{i=k}^l (a_{i,i+1}+b_{i-1,i})=a_{l,l+1}+b_{k-1,k}+\sum_{i=k}^{l-1}(a_{i,i+1}+b_{i,i+1})\\
&=a_{l,l+1}+b_{k-1,k}+\sum_{i=k}^{l-1} R_{i,i+1}\geqslant \sum_{i=k}^{l-1} R_{i,i+1}.
\end{align*}
For the same reason, we also have $\sum_{i=k}^l q_i\geqslant \sum_{i=k}^{l-1} R_{i,i+1}.$ Thus part (2) of the strengthened H.P.-condition holds.

\noindent\emph{Part II.} Assume $\lambda$ satisfies the strengthened H.P.-condition. We claim the following two inequalities hold:
\begin{gather}
 \sum_{i=k}^{l-1} R_{i,i+1}\leqslant
\min\{p_k,q_{k+1}\}+\sum_{i=k+1}^{l-1} p_i+\min\{p_l,q_{l-1}\};\label{stronger.ver.DI.equiv.con.1}\\
 \sum_{i=k}^{l-1} R_{i,i+1}\leqslant
\min\{q_k,p_{k+1}\}+\sum_{i=k+1}^{l-1} q_i+\min\{q_l,p_{l-1}\}.\label{stronger.ver.DI.equiv.con.2}\tag{$\ref{stronger.ver.DI.equiv.con.1}^\prime$}
\end{gather}
It suffices to prove the first line. When $l-k=1$, (\ref{stronger.ver.DI.equiv.con.1}) is exactly (\ref{Sing<Overlap}). 

When $l-k=2$, there are two summands $R_{k,k+1}$ and $R_{k+1,k+2}$. Suppose $q_{k+1}\geqslant p_k$ and $q_{k+1}\geqslant p_{k+2}$. Then (\ref{stronger.ver.DI.equiv.con.1}) becomes $R_{k,k+1}+R_{k+1,k+2}\leqslant p_k+p_{k+1}+p_{k+2}$, which is covered by (\ref{strengthened.H.P.(2)}). Suppose $q_{k+1}< p_k$ or $q_{k+1}<p_{k+2}$. We assume $q_{k+1}<p_k$ and keep on the proof. Then by (\ref{R12+R23<n}) of Lemma \ref{HP.condition.Rij}, we have
\begin{align*}
	R_{k,k+1}+R_{k+1,k+2}&\leqslant n_{k+1}=\min\{p_k,q_{k+1}\}+p_{k+1}
	\\
	&\leqslant\min\{p_k,q_{k+1}\}+p_{k+1}+\min\{p_{k+2},q_{k+1}\}.
\end{align*}

When $l-k>3$, we first prove a weaker version, which is
\begin{equation}\label{stronger.ver.DI.equiv.con.3}
	\sum_{i=k}^{l-1} R_{i,i+1}\leqslant
	\min\{p_k,q_{k+1}\}+\sum_{i=k+1}^l p_i. \tag{$\ref{stronger.ver.DI.equiv.con.1}^{\prime\prime}$}
\end{equation}
Suppose $p_k\leqslant q_{k+1}$, then (\ref{stronger.ver.DI.equiv.con.3}) reduces to (\ref{strengthened.H.P.(2)}). Suppose $p_k>q_{k+1}$. We add up the following two inequalities from (\ref{R12+R23<n}) of Lemma \ref{HP.condition.Rij} and (\ref{strengthened.H.P.(2)}),	
\begin{gather*}
	R_{k,k+1}+R_{k+1,k+2}\leqslant n_{k+1};\\
	R_{k+2,k+3}+\cdots+R_{l-l,l}\leqslant p_{k+3}+\cdots+p_{l}.
\end{gather*}
Then (\ref{stronger.ver.DI.equiv.con.3}) is proved. Repeat this process on (\ref{stronger.ver.DI.equiv.con.3}) and replace $p_l$ by $\min\{p_l,q_{l-1} \}$. Then (\ref{stronger.ver.DI.equiv.con.1}) is obtained.

Now we construct a solution $\{a_{i,i+1},b_{i,i+1}\}_{1\leqslant i\leqslant r-1}$ to (\ref{DI=0.Ineq.Red}) by induction on $r$. When $r=2$, it is easy to check that the following is a solution:
$$
a_{12}=\min\{p_1,q_2,R_{12}\},\quad b_{12}=R_{12}-\min\{p_1,q_2,R_{12}\}.
$$

For general cases when $r\geqslant 3$, we proceed with two subcases:

\noindent\emph{Subcase I.} Suppose $R_{r-1,r}\leqslant \min\{p_{r-1},q_{r-1},p_r,q_r\}$.

Consider the $A_{\fq_{(r-1)}}(\lambda_{(r-1)})$ where $\lambda_{(r-1)},\fq_{(r-1)}$ are defined as:
\begin{center}
	$\fq_{(r-1)}$ is attached to $\{(p_1,q_1),\cdots,(p_{r-1},q_{r-1})\}$; \\
	$\lambda_{(r-1)}:=(
	\lambda_1,\cdots,\lambda_1|
	\cdots\cdots|
	\lambda_{r-1},\cdots,\lambda_{r-1}).$
\end{center}
Clearly, $A_{\fq_{(r-1)}}(\lambda_{(r-1)})$ is nonzero. By inductive assumption on $r$, there is a solution $\{a_{i,i+1},b_{i,i+1}\}_{ 1\leqslant i\leqslant r-2}$ to (\ref{DI=0.Ineq.Red}) for $A_{\fq_{(r-1)}}(\lambda_{(r-1)})$. In particular, we have
\begin{equation}\label{ineq.ab.r-1}
a_{r-2,r-1}\leqslant q_{r-1},\quad b_{r-2,r-1}\leqslant p_{r-1}.
\end{equation}
Take
$$
a_{r-1,r}=\min\{p_{r-1}-b_{r-2,r-1},R_{r-1,r}\},\quad b_{r-1,r}=R_{r-1,r}-\min\{p_{r-1}-b_{r-2,r-1},R_{r-1,r}\}.
$$
Apparently, $a_{r-1,r}+b_{r-1,r}=R_{r-1,r}$ and both $a_{r-1,r}$ and $b_{r-1,r}$ are non-negative. We aim to prove $\{a_{i,i+1},b_{i,i+1}\}_{ 1\leqslant i\leqslant r-1}$ is a solution to (\ref{DI=0.Ineq.Red}) for $A_{\fq}(\lambda)$. It is sufficient to prove the following four inequalities:
$$
\begin{cases}
(\text{E1})~a_{r-1,r}+b_{r-2,r-1}\leqslant p_{r-1}, \\
(\text{E2})~a_{r-2,r-1}+b_{r-1,r}\leqslant q_{r-1}, \\
(\text{E3})~b_{r-1,r}\leqslant p_r,\\
(\text{E4})~a_{r-1,r}\leqslant q_r.
\end{cases}
$$
The definition of $a_{r-1,r}$ ensures (E1). The assumption $R_{r-1,r}\leqslant \min\{p_{r-1},q_{r-1},p_r,q_r\}$ proves (E4). When $p_{r-1}-b_{r-2,r-1}\geqslant R_{r-1,r}$, (E2) is reduced to (\ref{ineq.ab.r-1}), and (E3) is trivial because $b_{r-1,r}=0$. When $p_{r-1}-b_{r-2,r-1}< R_{r-1,r}$, (E2) and (E3) become
$$
\begin{cases}
(\text{E2*})~R_{r-2,r-1}+R_{r-1,r}\leqslant n_{r-1}, \\
(\text{E3*})~R_{r-1,r}+b_{r-2,r-1}\leqslant p_{r-1}+p_r.
\end{cases}
$$
(E2*) is covered by Lemma \ref{HP.condition.Rij}. (E3*) is obtained by adding up (\ref{ineq.ab.r-1}) and $R_{r-1,r}\leqslant p_r$.

\noindent\emph{Subcase II}. Suppose $R_{r-1,r}>\min\{p_{r-1},q_{r-1},p_r,q_r\}$. Define
$$d=R_{r-1,r}-\min\{p_{r-1},q_{r-1},p_r,q_r\}.$$

We assume $\min\{p_r,q_{r-1}\}\leqslant \min\{p_{r-1},q_r\}$. Then $d=R_{r-1,r}-\min\{p_r,q_{r-1}\}$. Then by (\ref{Sing<Overlap}), we have $d\leqslant \min\{p_{r-1},q_r\}$. Therefore, both $p_{r-1}-d$ and $q_r-d$ are non-negative. Consider the following pair $(\lambda^\prime,\fq^\prime)$ defined by
\begin{center}
	$\fq^\prime$ is attached to $\{(p_1,q_1),\cdots,(p_{r-1}-d,q_{r-1}),(p_r,q_r-d)\}$; \\
	$\lambda^\prime=(
	\lambda_1,\cdots,\lambda_1|
	\cdots|\overbrace{\lambda_{r-1},\cdots,\lambda_{r-1}}^{n_{r-1}-d}|\overbrace{\lambda_r-d,\cdots,\lambda_r-d}^{n_r-d}
	).$
\end{center}
Here
$$
	R^\prime_{r-1,r}=(\lambda_r-d)-\lambda_{r-1}=R_{r-1,r}-d.
$$
(If $\min\{p_r,q_{r-1}\}< \min\{p_{r-1},q_r\}$, we change $\mathfrak q^\prime$ to be attached to $\{(p_1,q_1),\cdots,(p_{r-1},q_{r-1}-d),(p_r-d,q_r)\}$. The rest of the proof are mostly the same.)
It is not hard to see that $\lambda^\prime$ is still in the nice range for $\fq^\prime$ by checking (\ref{lambda.nice.condition.equiv}). We shall prove $A_{\fq^\prime}(\lambda^\prime)$ is nonzero and $\lambda^\prime$ satisfies the strengthened H.P.-condition for $\fq^\prime$. Once the proof is done, it is easy to see $A_{\fq^\prime}(\lambda^\prime)$ falls in subcase I.

In order to prove that $A_{\fq^\prime}(\lambda^\prime)$ is nonzero, it suffices to prove the following two inequalities from Theorem \ref{main.result} and (\ref{Sing<Overlap}):
\begin{gather}
	R^\prime_{r-1,r}\leqslant \min\{p_{r-1}-d,q_r-d \}+\min\{q_{r-1},p_r\},\label{subcaseII.I}\\
	R_{r-2,r-1}\leqslant \min\{p_{r-2},q_{r-1}\}+\min\{q_{r-2},p_{r-1}-d\}\label{subcaseII.II}.
\end{gather}
For (\ref{subcaseII.I}),
\begin{align*}
\min\{p_{r-1}-d,q_r-d \}+\min\{q_{r-1},p_r\}&=\min\{p_{r-1},q_r \}+\min\{q_{r-1},p_r\}-d\\
&\leqslant R_{r-1,r}-d=R^\prime_{r-1,r}.
\end{align*}
For (\ref{subcaseII.II}), it reduces to (\ref{Sing<Overlap}) when $q_{r-2}\leqslant p_{r-1}-d$. Thus we may assume $q_{r-2}> p_{r-1}-d$. By (\ref{stronger.ver.DI.equiv.con.1}),
$$
R_{r-2,r-1}+R_{r-1,r}\leqslant \min\{p_r,q_{r-1}\}+p_{r-1}+\min\{p_{r-2},q_{r-1}\}.
$$
Noticing that $R_{r-1,r}=d+\min\{p_r,q_{r-1}\}$, we have
$$R_{r-2,r-1}\leqslant (p_{r-1}-d)+\min\{p_{r-2},q_{r-1}\}=\min\{q_{r-2},p_{r-1}-d\}+\min\{p_{r-2},q_{r-1}\}.$$

In order to prove that $\lambda^\prime$ satisfies the strengthened H.P.-condition for $\fq^\prime$, it suffices to prove the following two inequalities:
\begin{gather}
	R_{r-2,r-1}+R^\prime_{r-1,r}\leqslant n_{r-1}-d,\label{Aq'lambda'I}\\
	R_{k,k+1}+\cdots+R_{r-2,r-1}+R^\prime_{r-1,r} \leqslant \min\left\{\sum_{i=k}^r p_i,\sum_{i=k}^r q_i\right\}-d. \label{Aq'lambda'II}
\end{gather}
For (\ref{Aq'lambda'I}),
$$R_{r-2,r-1}+R^\prime_{r-1,r}=R_{r-2,r-1}+R_{r-1,r}-d\leqslant n_{r-1}-d.$$
For (\ref{Aq'lambda'II}), it suffices to prove
$$
R_{k,k+1}+\cdots+R_{r-2,r-1}+R_{r-1,r}^\prime\leqslant p_k+\cdots+p_r-d.
$$
By (\ref{stronger.ver.DI.equiv.con.1}), we have
$$
R_{k,k+1}+\cdots+R_{r-2,r-1}+R_{r-1,r}\leqslant p_k+\cdots+p_{r-1}+\min\{p_r,q_{r-1}\}.
$$
Then (\ref{Aq'lambda'II}) is implied by the inequality above and the fact $R_{r-1,r}-\min\{p_r,q_{r-1}\}=d$.

At this point, we see $A_{\fq^\prime}(\lambda^\prime)$ is nonzero and falls in subcase I. Let $\{a_{i,i+1}^\prime,b_{i,i+1}^\prime\}_{1\leqslant i\leqslant r-1}$ be a solution to (\ref{DI=0.Ineq.Red}) for $A_{\fq^\prime}(\lambda^\prime)$. Then let $\{a_{i,i+1},b_{i,i+1}\}$ be identical to $\{a_{i,i+1}^\prime,b_{i,i+1}^\prime\}$ except $a_{r-1,r}$. Take $a_{r-1,r}=a_{r-1,r}^\prime+d$. Compare the two versions of (\ref{DI=0.Ineq.Red}) from $A_{\mathfrak{q}^\prime}(\lambda^\prime)$ and $A_\mathfrak{q}(\lambda)$. We see that only three places are different:
$$
A_{\mathfrak{q}^\prime}(\lambda^\prime)~\text{version:}
\begin{cases}
	(a_{r-1,r}-d)+b_{r-1,r}=R_{r-1,r}-d,\\
	(a_{r-1,r}-d)+b_{r-2,r-1}\leqslant p_{r-1}-d,\\
	a_{r-1,r}-d\leqslant q_{r-1}-d;
\end{cases}
$$
$$
A_{\mathfrak{q}}(\lambda)~\text{version:}
\begin{cases}
a_{r-1,r}+b_{r-1,r}=R_{r-1,r},\\
	a_{r-1,r}+b_{r-2,r-1}\leqslant p_{r-1},\\
	a_{r-1,r}\leqslant q_{r-1}.
\end{cases}
$$
They are obviously equivalent. Therefore, $\{a_{i,i+1},b_{i,i+1}\}_{1\leqslant i\leqslant r-1}$ is a solution to (\ref{DI=0.Ineq.Red}) for $A_\mathfrak{q}(\lambda)$.

By Lemma \ref{DiracIndexNonzero}, we know $\DI(A_\fq(\lambda))\neq 0$.
\end{proof}
\end{thm}

\begin{cor}
Let a nonzero $A_\fq(\lambda)$ be in the nice range, and $\fq$ be attached to the sequence $\{(p_1,q_1),\cdots,(p_r,q_r)\}$. When $r\leqslant 3$, $H_D(A_\fq(\lambda))\neq 0$ if and only if the H.P.-condition stated in Lemma \ref{HP-condition} holds.
\begin{proof}
	It suffices to prove that the H.P.-condition is a sufficient condition for $H_D(A_\fq(\lambda))\neq 0$. When $r\leqslant 3$, the strengthened H.P.-condition is covered by the H.P.-condition and the nonzero criterion for $A_\fq(\lambda)$. By Theorem \ref{DI.equiv.condition.thm}, $\DI(A_\fq(\lambda))\neq 0$. Hence, $H_D(A_\fq(\lambda))\neq 0$.
\end{proof}
\end{cor}

\noindent\emph{Remark.} Due to the definition of the Dirac index, it is possible that the positive and negative parts of the Dirac cohomology get completely canceled in the Dirac index. The corollary above indicates that there is no such cancellation in the case of $A_\fq(\lambda)$ of $U(p,q)$ when the structure of $\fq$ is relatively simple. In the general case when $r>3$, we still believe that the Dirac index will not completely canceled out.

\end{document}